\documentclass[11pt]{article}
\usepackage[dvips]{graphicx}%
\usepackage{textcomp}
\usepackage{amsbsy}
\usepackage{latexsym}
\usepackage[mathscr]{eucal}
\usepackage{amsfonts,amsmath,amsthm}
\usepackage{amssymb}
\usepackage{xcolor}

\usepackage{mathtools}
\usepackage[symbol]{footmisc}

\usepackage{fullpage}

\def\rnew{\color{magenta}}
\def\cnew{\color{blue} }
\newtheorem{lemma}{Lemma}[section]
\newtheorem{prop}[lemma]{Proposition}
\newtheorem{cor}[lemma]{Corollary}
\newtheorem{theorem}[lemma]{Theorem}
\newtheorem{remark}[lemma]{Remark}

\newtheorem{definition}[lemma]{Definition}

\def\RR{\rm \hbox{I\kern-.2em\hbox{R}}}
\def\NN{\rm \hbox{I\kern-.2em\hbox{N}}}
\def\ZZ{\rm {{\rm Z}\kern-.28em{\rm Z}}}
\def\CC{\rm \hbox{C\kern -.5em {\raise .32ex \hbox{$\scriptscriptstyle
|$}}\kern
-.22em{\raise .6ex \hbox{$\scriptscriptstyle |$}}\kern .4em}}

\def\<{\langle}
\def\>{\rangle}
\def\t{\tilde}

\def\e{\varepsilon}

\def\nl{\newline}
\def\o{\overline}
\def\wt{\widetilde}

\def\cT{{\cal T}}
\def\cA{{\cal A}}
\def\cI{{\cal I}}

\def\cB{{\cal B}}

\def\cM{{\cal M}}

\def\Chi{\raise .3ex
\hbox{\large $\chi$}} 
\def\lsima{\hbox{\kern -.6em\raisebox{-1ex}{$~\stackrel{\textstyle<}{\sim}~$}}\kern -.4em}
\def\lsim{\hbox{\kern -.2em\raisebox{-1ex}{$~\stackrel{\textstyle<}{\sim}~$}}\kern -.2em}
\def\[{\Bigl [}
\def\]{\Bigr ]}
\def\({\Bigl (}
\def\){\Bigr )}
\def\[{\Bigl [}
\def\]{\Bigr ]}
\def\({\Bigl (}
\def\){\Bigr )}

\def\R{\mathbb{R}}

\def\nl{\newline}
\def\T{{\relax\ifmmode I\!\!\hspace{-1pt}T\else$I\!\!\hspace{-1pt}T$\fi}}
\def\N{\mathbb{N}}
\def\Z{\mathbb{Z}}
\def\N{\mathbb{N}}

\def\lsim{\raisebox{-1ex}{$~\stackrel{\textstyle<}{\sim}~$}}

  \def\NN{N}                  %

\def\t#1{\tilde{#1}}

\def\cA{{\cal A}}

\def\cA{{\cal A}}
\def\cC{{\cal C}}
\def\cS{{\cal F}}

\def\cI{{\cal I}}

\def\cM{{\cal M}}
\def\cN{{\cal N}}
\def\cT{{\cal T}}

\def\cL{{\cal L}}
\def\cB{{\cal B}}

\def\cK{{\cal K}}
\def\cX{{\cal X}}
\def\cS{{\cal S}}

\def\cR{{\cal R}}

\newcommand{\be}{\begin{equation}}
\newcommand{\ee}{\end{equation}}
\newcommand{\bea}{$$ \begin{array}{lll}}
\newcommand{\eea}{\end{array} $$}

\def \exp{\mathop{\rm    exp}}

\newcommand{\beqn}{\begin{equation}}
\newcommand{\eeqn}{\end{equation}}

\makeatletter
\@addtoreset{equation}{section}
\makeatother

\newcommand\dist{\mathop{\rm dist}}

\newcommand\eref[1]{{\rm (\ref{#1})}}
\newcommand{\iref}[1]{{\rm (\ref{#1})}}

\newcommand{\Funding}{This research was supported in part by the NSF Grants DMS-2409807 (AB), DMS-2513112 (WD), DMS-2245097 (WD), DMS-2038080 (WD), DMS-2134140 (RD, GP), DMS-2424305 (JWS), by the Foundation Sciences Math\'ematiques de Paris (AC), by the German Research Foundation grant CRC 1481 (WD), and by the ONR MURI Grant N00014-20-1-2787 (RD, GP, JWS).}

\newcommand{\Addresses}{{
  \bigskip
  \footnotesize

  A.~Bonito, R. DeVore, G. Petrova, and J.W. Siegel, \textsc{Department of Mathematics, Texas A\&M University, College Station, TX 77843, USA;}
  \texttt{\{bonito,ronalddevore,gpetrova,jwsiegel\}@tamu.edu}

  \medskip

  A. Cohen, \textsc{Laboratoire Jacques-Louis Lions, Sorbonne Universit\'e, 4, Place Jussieu, 75005 Paris, France;}
  \texttt{albert.cohen@sorbonneuniversite.fr}

  \medskip

  W.~Dahmen, \textsc{Department of Mathematics, University of South Carolina, Columbia, SC 29208, USA;} \texttt{dahmen@math.sc.edu}

}}

\title{Sampling and reconstruction of convex functions\thanks{\Funding}}

\author{Andrea Bonito \and Albert Cohen \and Wolfgang Dahmen \and Ronald DeVore \and Guergana Petrova \and Jonathan W. Siegel}

\date{\today}

\begin{document}

\maketitle

\begin{abstract}
We discuss optimal recovery for classes of multivariate convex functions
from given point samples, as well as the sampling numbers of these classes,
corresponding to optimal sample choices.  Upper and lower bounds for either variant are established when
the reconstruction error is measured in $L_p$ for $1\leq p\leq \infty$.
These bounds match, sometimes up to logarithmic factors, and therefore characterize
the respective optimal rate of decay. 
For classical smoothness classes such as Sobolev, H\"older or Besov spaces,
it is well known that the optimal decay rate of sampling numbers can be achieved by sampling
on uniform tensor product grids and using  linear methods of reconstruction, such as piecewise polynomial
interpolation. One of the main findings in this paper is that for classes of convex functions, these 
procedures generally produce suboptimal rates, except when  $p=1$ and 
$p=\infty$, and
are outperformed by nonlinear reconstruction methods that do not employ tensor product grids.
\end{abstract}

\vspace{0.5cm}

\section{Introduction}

\subsection{Recovery from point samples}
 
The problem of recovering a general multivariate function $f:\Omega\to \R$,  defined on a closed set $\Omega\subset \R^d$, 
from point samples is ubiquitous in a large number of applications. Given a set 
$\Lambda:=\{x^1,\dots,x^n\}$  of sample sites $x^j=(x^j_1,\dots,x_d^j)\in\Omega$, and the data vector 
\be 
\label{data} 
w=w(f):=(f(x^j))_{j=1}^n\in\R^n
\ee
of a continuous function $f\in C(\Omega)$, we are interested in constructing an approximation $\t f$ to $f$ with
recovery error $\|f-\t f\|_X$ measured in some Banach space norm $\|\cdot\|_X$.  In this paper, $X$ will always be an $L_p$ space with $1\le p\le \infty$,  and $\Omega$ will be the unit cube, that is,
\be
X=L_p(\Omega) \quad {\rm and} \quad \Omega=\Omega_d:=[0,1]^d.
\ee

In order to give  quantitative bounds on the recovery error in $X$,  one needs to make  more  assumptions about $f$ than the fact that $f$  is merely continuous.  Indeed, otherwise $f$ can be anything away from the data sites.
Such an additional  assumption is commonly referred to as a {\it prior} or {\it model class assumption} and consists of a requirement that $f\in \cK$, where
$\cK$ is a compact set in $X$.  Given the data sites $\Lambda$  and such a model class $\cK$,  a  recovery algorithm is  a mapping 
\be 
\label{alg}
A: \R^n\to X,
\ee
and the approximation to $f$ by  this algorithm is
\be 
\label{recovery}
\t f:= A(w(f)).
\ee
We denote by
\be
\label{allA}
\cA:=\{A:\R^n\to X\}
\ee
the set of all possible  mappings.  

The performance of an algorithm  $A$ on an individual $f$ is given by
\be 
\label{errorA}
E(f;\Lambda,A)_X:= \|f-A(w(f))\|_X,
\ee
whereas the worst case performance on the model class $\cK$ is 
\be
\label{errorAK}
E(\cK;\Lambda,A)_X:=\sup_{f\in \cK}E(f;\Lambda,A)_X.
\ee
The optimal recovery of the functions $f$ in the model class $\cK$ from the data observations $w(f)$ is then given by the quantity
\be 
\label{or}
\rho(\cK;\Lambda)_X:=\inf_{A\in \cA} E(\cK;\Lambda,A)_X,
\ee
that describes the performance of a recovery algorithm that is optimal over the class $\cK$ for the given data sites $\Lambda$.

There is a simple theoretical description of an optimal recovery algorithm and its performance error.  
Given a set $\Lambda$ of data sites and a  data vector $w \in \mathbb R^n$, let  $\cK_w$ denote the set of all functions $g\in \cK$ which satisfy the given data, i.e., $w(g)=w$.  Consider a ball $B_w=B(f_w,r_w)\subset X$
of smallest possible radius $r_w$ that contains $\cK_w$.  Then, the center $f_w$ of $B_w$ is an optimal recovery from the given data vector 
$w$,  and the radius $r_w$ is the optimal recovery error for the observed data vector $w$. 
Note that for those $w\in\R^n$ for which  $\cK_w$ is empty, one has $r_w=0$.
Any mapping $A^*=A^*(\cK,\Lambda)$ such that
\be 
\label{optA}
A^*:w\to f_w,
\ee
when $\cK_w$ is not empty, is an optimal recovery algorithm for $\cK$ and $\Lambda$,
and we can identify
\be
\label{opterror}
\rho(\cK;\Lambda)_X =\sup_{f\in \cK} \|f-A^*(w(f))\|_X=\sup_{f\in \cK}r_{w(f)}
\ee
as the optimal recovery error. 

The optimal recovery mapping $A^*$ is only of theoretical interest. In applications, one wants a more practical mapping.
Notice that if in   place of $A^*$ one constructs a mapping $\tilde A$ that, when $\cK_w$ is not empty, maps each $w$ into an element $g_w\in \cK_w$, i.e.,   finds an element in the model class which satisfies the measurements, 
then the corresponding algorithm is {\it near optimal}.  Namely, we have
\be
\label{nor}
\|f-\tilde A(w(f))\|_X\le 2 \rho(\cK;\Lambda)_X, \quad f\in \cK.
\ee
The maps $A^*$ and $\t A$ are actually optimal, respectively near-optimal, 
not only over the class $\cK$, but also over the set $\cK_w$ for each individual possible data instance $w$.
This property is usually referred to as {\it instance optimality}.
 
 Note that the optimal recovery framework does not impose constraints on the recovery map $A$. Nevertheless, in practical applications, it is often preferable to restrict $A$ to be {\it linear}. This leads to the definition of linear optimal recovery $\rho^\ell(\cK;\Lambda)_X$, where in \eref{or} the infimum  is taken over the class $\cA^\ell$ of all linear maps from $\R^n$ to $X$.
Obviously, one has that 
\be
\label{lr}
\rho(\cK;\Lambda)_X\le \rho^\ell(\cK;\Lambda)_X.
\ee

In both cases, a crucial question that arises  is the best choice of data sites.  For a fair competition, we fix the number $n$ of the data sites and define
\be
\label{serror}
\rho_n(\cK)_X:= \inf_{\#(\Lambda)=n} \rho(\cK;\Lambda)_X,
\ee
which are called the {\it  sampling numbers} for the model class $\cK$ in the norm $\|\cdot\|_{X}$. One similarly defines the {\it linear  sampling numbers} 
\be
\label{lerror}
\rho_n^\ell(\cK)_X:= \inf_{\#(\Lambda)=n} \rho^\ell(\cK;\Lambda)_X.
\ee 
There is a practical interest in evaluating sampling numbers.
Indeed, minimizing the number of point evaluations needed for approximating functions from a given class with some prescribed accuracy is particularly important when each evaluation is costly (it may for example necessitate solving a PDE, or running a physical or numerical experiment).

Sampling numbers can be related to classical quantities that describe the complexity of compact sets. The first classical notion are the {\it Kolmogorov widths}
\begin{equation}\label{d:kol}
d_n(\cK)_X:=\inf_{\dim(X_n)\leq n}\max_{f\in \cK}\min_{g\in X_n}\|f-g\|_X,
\end{equation}
which describe the best possible approximation of $\cK$ by an $n$-dimensional linear space. Obviously, one has the lower bound
\be
d_n(\cK)_X\leq \rho_n^\ell(\cK)_X.
\label{widthsamp}
\ee
The second classical notion are the {\it entropy numbers}
\be
\e_n(\cK)_X:=\inf\{\e \; : \; \exists f_1,\dots,f_{2^n}\in X \; {\rm s.t.}\;
\max_{f\in \cK}\min_i \|f-f_i\|_X \leq \e\},
\ee
which, for every fixed $n\in\N$, are  the smallest $\e$ such that $\cK$ is covered by $2^n$ balls of radius $\e$. It has been shown in \cite{CDPW} that the rate of decay of the worst case error in approximating  $\cK$  by any stable, linear or non-linear, approximation method using $n$ parameters cannot be faster than the rate of decay of the entropy numbers. Therefore, we do not expect $\rho_n(\cK)_X$ to decay faster than $\e_n(\cK)_X$.

The sampling numbers and linear sampling numbers have been extensively studied in the literature. 
In particular, their asymptotic rate of decay  have been determined for many classical model classes.
For example, if $\cK$ is the unit ball of a Lipschitz, Sobolev, or Besov space which is embedded in $C(\Omega)$, one   knows the asymptotics of the  sampling numbers $\rho_n(\cK)_{L_p}$ and linear sampling numbers $\rho^\ell_n(\cK)_{L_p}$, and their decay rates agree.  Moreover, it is known that, choosing the data sites 
on uniform tensor product grids, realizes this optimal rate of decay. We refer to \cite{NT}, see Theorem 23, for results of this type,
or
 \cite{KU}, see  Theorem 11.1,  for a complete account on various notions of widths
for Sobolev classes, including their sampling numbers.

\subsection{Model classes of convex functions}

In some applications, one is interested not only in classical smoothness classes but in 
model classes which impose different natural constraints for membership.  
The present paper is interested in one such setting where the model classes consist of convex functions,  which are ubiquitously encountered in optimization tasks coming from various applications. 

We denote by 
$\cC_{\rm conv}(\Omega)$ 
the set of all   convex functions on $\Omega$.
Recall that a function $f$ is  in   $\cC_{\rm conv}(\Omega)$, if and only if its restriction $f|_L$ to each line segment $L\subset\Omega$ is a convex univariate function. Note  that every  
$f\in \cC_{\rm conv}(\Omega)$ is 
continuous on ${\rm Int}(\Omega):=(0,1)^d$,  but  generally not in
$C^1$. 

We use the standard notation $v\cdot w$ for the scalar product of $v,w\in \R^d$, and $|v|$ for the Euclidean norm of $v$.
At each point $x\in {\rm Int}(\Omega)$, 
the set 
\begin{equation}
\label{subgradient}
    \partial f(x) := \{v\in \R^d:~f(y) \geq f(x) + v\cdot (y-x)~\text{for all $y\in \Omega $}\},
\end{equation} is called the  {\it subgradient}  of $f$ at $x$. 
The set $\partial f(x)$ 
 is non-empty, convex, and compact. Moreover, 
an important property of the subgradient is that it is \textit{monotone} in the sense that whenever $x,y\in \Omega$, $v\in \partial f(x)$, and $w\in \partial f(y)$, then
\begin{equation}
    (v - w)\cdot (x - y) \geq 0.
\end{equation}

For a univariate convex function $g:[a,b]\to \R$,
its 
subgradient is the interval
\be
\partial g(t)=[g'_-(t),\,g'_+(t)],
\ee
and each of the functions $g'_-$ and $g'_+$ is monotone nondecreasing. More precisely, we have that
\be
\label{1d}
g'_-(t_1)\leq g'_+(t_1)\leq g'_-(t_2)\leq g'_+(t_2), \quad \hbox{for every}\,\, a<t_1<t_2<b.
\ee
Also $g'_-(t)$ has a limit at $t\to a^+$.  We define $g'_-(a)$ to be this limit (which could be $-\infty$).  We define $g'_+(b)$ (which could be $+\infty$) similarly.

We consider  two sub-classes of $\cC_{\rm conv}(\Omega)$ imposing the following  restrictions:
\begin{itemize}
\item 
Uniform bound on the functions: 
\be
\cB=\cB(\Omega):=\{f\in \cC_{\rm conv}(\Omega): \; |f(x)|\leq 1, \; x\in \Omega\}.
\ee
\item 
Uniform bound on the subgradient and the function: 
\be
\label{defL1}
\cL=\cL(\Omega):=\cB\cap\{f\in \cC_{\rm conv}(\Omega): \; \forall v\in \partial f(x),  \; x\in {\rm Int}(\Omega),\ {\rm we \ have}\  |v|\le 1 \}.
\ee
\end{itemize}
Obviously, one has 
$\cL\subsetneq\cB$. It is 
easily checked that the class $\cL$ can equivalently be defined by the  Lipschitz condition:
\be
\cL=\{f\in \cC_{\rm conv}(\Omega): |f(x)|\leq 1, \,\,\;  
|f(x)-f(y)|\leq  |x-y|, \; x,y\in \Omega\}.
\ee

\subsection{Outline of the paper}

The  main results in this paper provide matching upper and lower bounds (up to logarithmic factors) for the sampling numbers of the model classes $\cB$ and $\cL$. We also determine the rates of the
optimal recovery error when the data sites are restricted to be tensor product grids.
Our main finding is that, for both classes $\cB$ and $\cL$, the sampling rates cannot in general be achieved if one imposes either the restriction that the sampling grid is a tensor product grid, or the restriction that the recovery is given by a linear mapping $A$. Specifically, except when $p=1$ or $p=\infty$, the decay rates for the sampling numbers cannot be achieved using tensor product grids. Further, when $p = 2$, the decay rates of the linear sampling numbers are substantially worse than those of their nonlinear counterparts.

We stress that, while the classes $\cB$ and $\cL$ are
contained in classical
smoothness classes $\cK$, such as balls of $W^s(L_p)$ for certain ranges of $s$ and $p$, their
sampling numbers cannot generally be derived from these embeddings.

We begin this paper by studying the setting where the set of sample points is a tensor product grid. For the class $\cL$, we show in \S \ref{S:tensor} that, when $1\le p\le \infty$,
uniform tensor product grids $\Lambda_n$ with $n=\#(\Lambda_n)$, give the optimal recovery rate 
\be
\label{firstrate}
\rho(\cL,\Lambda_n)_{L_p}\asymp n^{-r/d},\quad n\ge 1.
\ee
where
\be
r:=r(p)=1+\frac 1 p.
\ee
This rate is also proved
to be optimal even when the competition is allowed to take place over all possible
tensor product grids (not necessarily uniform) consisting of $\asymp n$ data sites and can be achieved by a simple linear recovery method.
Our analysis is based on approximation error estimates for multilinear interpolation 
of convex functions, which seem to be new,  and could be of interest in numerical analysis. 
It follows from these results that one has
\be
\rho_n(\cL)_{L_p}\leq \rho_n^\ell(\cL)_{L_p} \lsim n^{-r/d}.
\label{upperrate}
\ee

We also establish lower bounds for the sampling numbers of $\cL$. In the univariate case, they 
match the upper bounds, 
and thus the rate \eref{firstrate} is optimal
\be
\rho_n(\cL)_{L_p}\asymp n^{-r/d}, \quad d=1, \quad 1\leq p\leq \infty.
\ee
When $d\geq 2$, the rate in \iref{firstrate} is also optimal for the sampling numbers when $p=1$ or $p=\infty$, namely,
\be
\rho_n(\cL)_{L_1}\asymp  \rho_n^\ell(\cL)_{L_1}\asymp n^{-2/d} \quad {\rm and} \quad \rho_n(\cL)_{L_\infty}\asymp  \rho_n^\ell(\cL)_{L_\infty}\asymp n^{-1/d}.
\ee
It is also optimal for the linear sampling numbers when $p=2$, that is,
\be
\rho_n^\ell(\cL)_{L_2}\asymp  n^{-\frac 3 {2d}}.
\ee
It turns out, as described below, that for other values of $p$, determining the asymptotic decay of the sampling numbers of $\cL$ is quite subtle and cannot be obtained by simple constructions using tensor
product grids or linear recovery.

For the larger class $\cB$, we show in \S \ref{S:TensorB} that, restricting to tensor product grids, the rates \eqref{firstrate} also hold for $1\leq p < \infty$, with the notable difference that the tensor product sampling grid which achieves this rate can no longer be taken as uniform, but must be chosen 
  with a grading that depends on $p$. Note that the sampling numbers of the class $\cB$ do not decay to zero in
the  case $p=\infty$.

We   turn next in the paper to a more refined study of the optimal sampling rate 
$\rho_n(\cL)_{L_p}$. The rather surprising result, proved
in \S \ref{S:optimal}, is that, when $1<p<\infty$,
by using specific sets $\Lambda_n^*$ of data sites that are not of tensor product type, 
one achieves substantially better rates 
than the optimal rates for tensor product grids. More precisely, up to logarithmic
factors, we obtain the optimal rate for the sampling numbers
\be
\rho_n(\cL)_{L_p} \asymp n^{-\t r/d},
\ee
where
\be
\t r:=\t r(p,d)=1+\min\Big\{\frac d p,1\Big\}=
\begin{cases}
2, &\,1\leq p\leq d,\\
1+\frac{d}{p}, &\,p\geq d.
\end{cases}
\ee
For $1<p<\infty$ and $d\geq 2$, this optimal rate thus improves on the rate of decay $n^{-r/d}$ achievable 
by tensor product grids, as well as on the rate of decay of the linear sampling numbers when $p=2$.
We extend these results to the larger class $\cB$ in \S \ref{S:optimalB}, reaching the same optimal rate of decay $n^{-\t r/d}$, except in the case $p=\infty$ where, as we have already mentioned, the sampling numbers do not converge to zero.

The proof of the optimal rates
in \S \ref{S:optimal} and \S \ref{S:optimalB}
is not fully constructive. We prove the existence of sampling sets $\Lambda^*_n$ of cardinality $n$ and of optimal recovery maps $A$ that achieve an error bound (up to logarithmic factors)
\be
E(\cK,\Lambda^*_n,A)_{L_p} \lesssim n^{-\t r/d},
\ee
when  $\cK=\cL$, or $\cK=\cB$. 
We present in \S \ref{S:OR} 
computationally feasible near optimal 
recovery maps $A_{\cL}$ and $A_{\cB}$
that are inherently nonlinear and achieve such error bounds. We conclude the paper 
with  \S 7, where some remarks 
involving widths, entropy numbers, and open questions are discussed.

In what follows, we will use 
$C$ to denote constants in our upper bounds,  and $c$ to denote constants in our lower bounds. These  constants depend only on the dimension $d$ (and sometimes also on $p$), and their values can change from formula to formula. We also use the notation 
$A\asymp B$ when $cA\leq B\leq CA$, or $B\lesssim A$, if $B\leq CA$.

\section{  Sampling on tensor product grids }
\label{S:tensor}

One of the most natural choices for the set $\Lambda$ of sampling sites 
is  a tensor product grid.
Suppose that $\Gamma_1,\dots,\Gamma_d$ are sets with  $m_1,\dots,m_d$   points in $[0,1]$, respectively.  Then, the tensor product grid 
\be 
\label{tpgrid}
\Lambda=\Lambda_n:=\Gamma_1\times\cdots\times \Gamma_d
\ee
is a set of $n=m_1\cdots m_d$ sampling sites from  $\Omega$.  A special case is the {\it uniform grid} 
 $\Lambda_n:=\Gamma_1\times\cdots\times \Gamma_d$ with $\Gamma_j:=\Gamma:=\{0,1/m,2/m,\dots,1\}$ for all $j=1,\dots,d$, and $n=(m+1)^d$.
 
We let $\cT_n$  denote the class of all tensor product grids consisting of at most $n$ sampling sites. 
In this section, we  analyze the performance of sampling on such tensor product grids for the recovery of functions in the model class $\cL$. 

\begin{remark} 
\label{remdecrease}
Note that when imposing tensor product grids, not all possible values of $n\geq 1$ appear in their cardinalities. Since 
the quantities $\rho_n(\cK)_X$ and $\rho_n^\ell(\cK)_X$ are obviously non-increasing,  for establishing an algebraic rate of decay $n^{-\alpha}$ on these quantities, it is enough to show this 
rate for all $n$ in a subsequence $(n_k)_{k\geq 1}$ such that $\frac {n_{k+1}}{n_k}$ is uniformly bounded. 
\end{remark}

 \subsection {Recovery of a function from a tensor product grid}

\label{SS:pwlinterpolation}
Suppose that $\Lambda$ is a set of data sites given by a tensor product grid of the form \eref{tpgrid}, and further let $\cK$ be a model class.  Given data consisting of the values of a  function $f\in \cK$ at these data sites,  we are interested in constructing a {\em practical} recovery map $A$ that achieves near optimal performance for $\cK$ given that $\Lambda$ has been fixed.  A natural choice for $A$ is to use the data observations to create a function $g$ that interpolates the data.  While there are many possibilities for such interpolants, we put forward in this section a form of piecewise multi-linear interpolation, which we go on to show is near optimal for tensor grid sampling whenever $\cK$ is one of our two model classes $\cL$ or $\cB$.

To precisely describe the multi-linear interpolation map that we will use, let us first consider the case $d=1$.  In this case the data sites are a set $\Lambda:=\{0=\xi_0<\xi_1<  \cdots <\xi_m = 1\}$ of $m+1$ points from the interval $[0,1]$ and our data consists of the values of a function $f$ at these data sites.  We define
\be 
\label{intd1}
I_\Lambda^1(f):\ [0,1]\mapsto \R,
\ee
to be the continuous piecewise linear function on $[0,1]$ which interpolates the values of $f$ at the data sites and has its only breakpoints at the data sites. Note that $I_\Lambda^1(f)$ is equal to
\be
    I_\Lambda^1(f) = \sum_{i=0}^m f(\xi_i)\ell_{\Lambda,i},
\ee
where $\ell_{\Lambda,i}$ are the unique piecewise-linear (hat) functions, with breakpoints only at the datasites $\Lambda$, such that $\ell_{\Lambda,i}(\xi_j) = \delta_{ij}$.

When $d>1$ and $\Lambda=\Gamma_1\times \cdots \times\Gamma_d$, where the $\Gamma_j := \{0 = \xi_0^i < \cdots < \xi^j_{m_j} = 1\}$ are one-dimensional data sites (potentially different for each $i$), we define the interpolating function $I_\Lambda^d(f)$ via
\be
    I_\Lambda^d(f) := \sum_{i_1=0}^{m_1} \cdots \sum_{i_d=0}^{m_d}f(\xi_{i_1},...,\xi_{i_d})\ell_{\Lambda,i_1,...,i_d},
\ee
where the Lagrange functions are given by the tensor product
\be
    \ell_{\Lambda,i_1,...,i_d}(x_1,...,x_d) := \prod_{j=1}^d\ell_{\Gamma_j,i_j}(x_j).
\ee
Note that the operator $I_\Lambda^d$ is simply the standard $d$-dimensional multi-linear interpolation operator on a tensor product grid.

We will use that these interpolants can also be characterized inductively as follows:
\begin{itemize}
    \item Suppose that $I_{\tilde \Lambda}^{d-1} (\tilde f)$ has been defined for every tensor product grid $\tilde \Lambda$ of $[0,1]^{d-1}$ and every function $\tilde f$ defined on $\Omega_{d-1}=[0,1]^{d-1}$.
    \item Let $\tilde \Lambda:=\Gamma_1\times\cdots \times \Gamma_{d-1}$, which is a tensor product grid for $[0,1]^{d-1}$. Note that $\Lambda = \tilde \Lambda \times \Gamma_d$, and 
\be 
\label{defId}
I^d_\Lambda(f)(x) = I^1_{\Gamma_d}\(I^{d-1}_{\tilde\Lambda}(f(\cdot,t))(\tilde x)\)(x_d), \quad x=(\tilde x,x_d),
\ee
where the one-dimensional interpolation $I_{\Gamma_d}^1$ is applied to the univariate function
\be
    t \rightarrow I^{d-1}_{\tilde\Lambda}(f(\cdot,t))(\tilde x).
\ee
\end{itemize}

 \subsection{Main results for sampling on tensor product grids for the model class $\cL$}
 \label{SS:maintp}

We recall the notation
\be
r=r(p):=1+\frac{1}{p}, \quad 1\leq p
\leq \infty.
\ee
Our main results in this section are given in  the following theorem.
\begin{theorem}
\label{T:tensorL}
For the model class $\cL=\cL(\Omega_d)$,  we have the following results :
\vskip .1in
\noindent
{\rm (i)} If  $\Lambda_n \subset \Omega_d$, $d\geq 1$, is a uniform tensor product grid and $n=(m+1)^d$,
 then for $1\le p\le\infty$,    we have
\be
\label{ubL1}
\rho(\cL, \Lambda_n)_{L_p}\le  \rho^\ell(\cL, \Lambda_n)_{L_p}\le Cn^{-r/d},
\ee
with the constant $C$ depending only on $d$.  It follows that for $1\leq p\leq \infty$ and for every $n\geq 1$, 
\be
\label{sampl}
\rho_n(\cL)_{L_p}\le  \rho_n^\ell(\cL)_{L_p}\le Cn^{-r/d}.
\ee
\vskip .1in
\noindent
{\rm (ii)} For any tensor product grid $\Lambda \in\cT_n\subset \Omega_d$, $d\geq 1$, any $1\le p\le \infty$, and any $n\ge1$, we have the lower bounds
\be
\label{orL1}
\rho^\ell (\cL,\Lambda)_{L_p} \ge \rho (\cL,\Lambda)_{L_p}\ge c  n^{-r/d}, 
\ee  
where $c$ depends only on $d$.

\vskip .1in
\noindent
{\rm (iii)}  In the case $d = 1$ and $1 \leq p\leq \infty$
we have 
\be
\label{orLgen}
\rho_n(\cL)_{L_p}\asymp  \rho^\ell_n(\cL)_{L_p}\asymp n^{-r/d}.
\ee
Also, the estimate \eqref{orLgen} holds if $d > 1$ and $p = 1,\infty$.

\vskip .1in
\noindent
{\rm (iv)} When  $p=2$ and all $d\geq 1$, one has
\be
\label{orL2}
\rho^\ell_n(\cL)_{L_2}\asymp n^{-r/d}=n^{-\frac 3{2d}}.
\ee
The equivalency constants depend only on $d$.

\vskip .1in
\noindent
{\rm (v)}
When  $1<p< \infty$ and $d>1$, one has 
\be
\label{lowLgen}
\rho_n(\cL)_{L_p}\geq c n^{-\t r/d}, \quad \t r:=\t r(p,d)=\min\{2,1+d/p\}<r,
\ee
where $c$ depends only on $d$.
\end{theorem}

In other words, (i) and (ii) of this theorem show that  $n^{-r}$  is  the best recovery rate that one can obtain for the model class $\cL$ when one restricts the sampling sites to lie on a tensor product grid.  Moreover, this rate can be obtained using uniform sampling and a linear method of recovery.
This  obviously implies an upper bound $Cn^{-r}$ for the  rate of decay of $\rho_n(\cL)_{L_p}$ and $\rho^{\ell}_n(\cL)_{L_p}$ that is valid for all $d$ and $p$.   Item (iii) determines the asymptotic behavior
of $\rho_n(\cL)_{L_p}$, $n\ge 1$, but only in the special cases $d=1$ and $1\le p\le \infty$ and  for $d>1$ when  $p=1$ or $p=\infty$.  Item (iv) determines the asymptotic behavior of $\rho_n^\ell(\cL)_{L_2}$.  Item (v) gives a general lower bound $cn^{-\tilde r}$, $n\ge 1$, for $\rho_n(\cL)_{L_p}$ for all $1\le p\le \infty$.  Since $\tilde r>r$ this lower bound does not match the upper bound $Cn^{-r}$.  Later in this paper, we shall show that the upper bound for the decay rate can be improved to $n^{-\t r}$, save for logarithmic factors,  which in turn show that the correct decay rate is $n^{-\t r}$.  According to  Theorem \ref{T:tensorL}, this improved rate cannot be obtained using tensor product grids and according to (iv) cannot be obtained with a linear method of recovery.

\subsection{Upper bounds for uniform tensor product grids}
\label{SS:upper}

The upper bounds for the recovery error of  tensor grid sampling  given in Theorem \ref{T:tensorL}   are obtained by  uniform sampling and recovery using
multilinear interpolation as described in \S\ref{SS:pwlinterpolation}.  Therefore, in this section we restrict ourselves to analyzing this case.   We will  frequently use the following elementary lemma.
\begin{lemma}
\label{L:intlemma}
Let  $g: [a,b]\to \R$ be a univariate convex function and $\hat g$ be its linear interpolant at the endpoints of the interval, that is, $\hat g$ is the unique affine function such that 
$\hat g(a)=g(a)$ and $\hat g(b)=g(b)$. Then we have
\be\label{main}
\|g-\hat g\|_{L_p([a,b])}
\leq (b-a)^{1+1/p}(g'_-(b)-g'_+(a)),\quad 1\leq p\leq \infty.
\ee 
\end{lemma}
\begin{proof}
The function $\t g:=g-\hat g$ is 
convex, negative and satisfies $\t g(a)=\t g(b)=0$.
Since 
\be
\delta:=(g'_-(b)-g'_+(a))
=(\tilde g'_-(b)-\tilde g'_+(a))\geq 0,
\ee
one has
\be
0\leq \tilde g'_-(b)\leq \delta
\quad{\rm and}
\quad -\delta\leq \tilde g'_+(a)\leq 0.
\ee
Using the first inequality, the convexity of $\tilde g$, and the fact that $\tilde g(b)=0$, we find that
\be
0\geq \tilde g(x)
\geq \delta(x-b)\geq \delta(a-b),
\quad 
x\in [a,b],
\ee
which gives the $L_\infty$ bound. The $L_p$ bound follows by integration. Note that the bound \eref{main} could be improved by a multiplicative constant smaller than $1$, where the optimal constant corresponds to the extreme case when $\t g$ is piecewise affine
with slope $-\delta/2$ on
$[a,(a+b)/2]$ and $\delta/2$
on $[(a+b)/2,b]$.
\end{proof}

We now consider the multidimensional case where $d\in\{2,3,\dots\}$ and $\Omega=\Omega_d=[0,1]^d$.  Any point $x=(x_1,\dots,x_d)\in \Omega$ can be written as $x=(\tilde x_d,x_d)$ where $\tilde x_d\in [0,1]^{d-1}$ is the vector of its first $d-1$ components.
 We define the uniform tensor product grid
 $\Lambda_n:=\Lambda_n(d):=\Gamma_1\times\cdots\times \Gamma_d$ with $\Gamma_j:=\Gamma:=\{0,1/m,2/m,\dots,1\}$, $j=1,\dots,d$, consisting of  $n=(m+1)^d$ points and denote by $T_k$ the cell
\be\label{t-k-definition}
T_k:=[k_1h,(k_1+1)h]\times\cdots\times [k_dh,(k_d+1)h], \quad k=(k_1,\dots,k_d)\in \cS_m^d,
\ee
where 
\be
h:=m^{-1}
\asymp n^{-1/d},
\label{hmn}
\ee
and where we use the notation
\be
\cS_m:=\{0,\dots,m-1\},
\ee
and thus 
$\cS_m^d:=\{0,\dots,m-1\}^d$.

If $k=(\t k_d,k_d)$, where 
$\t k_d:=(k_1,\dots,k_{d-1})\in \cS_m^{d-1}$ and $k_d\in \cS_m$,
we have 
\be
\quad T_k=T_{\t k_d}\times [k_dh,(k_d+1)h], \quad  k\in
\cS_m^d.
\ee

Given $f\in\cL$ and the point values  
  of $f$  at the  data sites in $\Lambda_n$, we consider the multi-linear interpolant 
\be
I_h^df:=I^d_{\Lambda_n}f:\Omega_d\to \R, 
\ee
defined  in \S\ref{SS:pwlinterpolation}.
Given the relation \iref{hmn} between $h$ and $n$, the upper bound \eref{ubL1} will follow directly 
by proving the following lemma.
\begin{lemma}
\label{L:Tub} Let $1\le p\le\infty$ and $d\geq 1$.  Then, whenever $f\in \cL(\Omega_d)$, we have  
\be
\|f-I_h^df\|_{L_p}\leq 2dh^{r}.
\label{estIh}
\ee
\end{lemma}
\begin{proof}
    We only need to show \eref{estIh} for $p=1$ and $p=\infty$, since the general case
then follows from the H\"older's inequality
\be
\|f-I_h^df\|_{L_p}\leq \|f-I_h^df\|_{L_1}^{1/p}\|f-I_h^df\|_{L_\infty}^{1-1/p}.
\label{inter}
\ee

We first treat the case $d=1$,
and then deal with the 
case $d>1$ by an induction argument. In 
the univariate  case, one has  
\be
T_i=[ih,(i+1)h], \quad i\in\cS_m
\ee
and $I_h^1f$ is the piecewise linear interpolant,  which on each $T_i$ coincides with the linear interpolant
of $f$ at the end points 
$ih$ and $(i+1)h$. Assuming $f\in \cL(\Omega_1)$, Lemma \ref{L:intlemma}
says that on each interval $T_i$,
\be
\label{L1d1}
\|f-I_h^1f\|_{L_\infty(T_i)}
\leq   h(f'_-((i+1)h)-f'_+(ih))
\leq 2h.
\ee
Maximizing  over $i$, results in \iref{estIh}
for the value of  $p=\infty$. 

We next turn to bounding the error in $L_1$. Lemma \ref{L:intlemma}
also gives 
\be
\|f-I_h^1f\|_{L_1(T_i)}
\leq   h^2(f'_-((i+1)h)-f'_+(ih)),
\ee
and after summing  over $i$ and using the convexity of $f$, see \eref{1d}, we  obtain that
\be
\|f-I_h^1f\|_{L_1(\Omega_1)}
\leq h^2(f'_-(1)-f'_+(0)) \leq 2h^2,
\ee
which is \iref{estIh}
in the case $p=1$. Due to \eref{inter}, estimate \eref{estIh} is proven in the case $d=1$.

We turn now to the general case $d>1$.  We first  observe that  the restriction of a
multivariate convex function $f\in \cL(\Omega_d)$, obtained by freezing some of its  coordinates, 
remains a convex function of the rest of the  variables. Moreover,  at every point,  the restricted $f$ has all its 
subgradients  bounded by $1$. In particular, when we fix the first $(d-1)$ coordinates, the univariate function $t\mapsto g(t):=f(x_1, \ldots,x_{d-1},t)$ belongs to $\cL(\Omega_1)$. Let us assume that  \eref{estIh} holds for dimension $(d-1)$ and $1\le p\le \infty$ and seek to establish it for dimension $d$.
\nl
\nl
{\bf Case $p=1$:}   For $k=(k_1,\dots,k_d)\in \cS_m^d$, consider the cell
 $T_k$
and  let $x=(\t x_d,x_d)\in T_k$. We define
\be
\theta(x_d):=h^{-1}(x_d-k_dh)\in [0,1], \quad {\rm when}\quad x_d\in [k_dh, (k_d+1)h].
\ee
Using the inductive definition of $I^d_h f$, we can bound the interpolation error by
$$
\begin{array}{ll}
|f(x)-I^d_hf (x)|
& = |f(x)-(1-\theta(x_d)) [I^{d-1}_{h} f(\cdot,k_dh)](\t x_d)-\theta(x_d)
[I^{d-1}_{h} f(\cdot,(k_d+1)h)](\t x_d)| \\\\
& \leq  E_1(x)+E_2(x)+E_3(x).\\
\end{array}
$$
Here, the first error term $E_1$, is defined for $x=(\t x_d,x_d)\in  T_{\t k_d}\times [k_dh, (k_d+1)h]$ by
\be
E_1(x):=|f(\t x_d,x_d)-(1-\theta(x_d)) f(\t x_d,k_dh) -\theta(x_d) f(\t x_d,(k_d+1)h)|.
\ee
 It is the one dimensional interpolation error in the direction $x_d$ when using the exact values of $f$ on the faces $T_{\tilde k_d}$. The other two error terms  are
\be
E_2(x):=(1-\theta(x_d))|f(\t x_d,k_dh)-I^{d-1}_{h} [f(\cdot,k_dh)](\t x_d)|,
\ee
and
\be
E_3(x):=\theta(x_d) |f(\t x_d,(k_d+1)h)- I^{d-1}_{h} [f(\cdot,(k_d+1)h)](\t x_d)|.
\ee
They account for the error of $(d-1)$ dimensional interpolation of the 
convex functions $f(\cdot,k_dh)$ and 
$f(\cdot,(k_d+1)h)$, respectively.
In summary, we have
\be
\label{zero}
\|f-I_h^df\|_{L_1}\leq 
\sum_{k\in \cS_m^d} \int_{T_k}E_1
+\sum_{k\in \cS_m^d} \int_{T_k}E_2+\sum_{k\in \cS_m^d} \int_{T_k}E_3 = :\Sigma_1+\Sigma_2+\Sigma_3.
\ee
 We now bound the three sums.  For the first sum, we write
\be
\Sigma_1:=\sum_{k\in \cS_m^d}\int_{T_k} E_1(x)dx=\sum_{k\in \cS_m^d}\int_{ T_{\t k_d}} \left [\int_{k_dh}^{(k_d+1)h}
 E_1(\t x_d,x_d)dx_d\right ] d\t x_d,
\ee
and observe that for each fixed $\t x_d\in T_{\t k_d}$, the univariate integral is the $L_1$ error of  interpolation of the univariate function $f(\t x_d,\cdot)\in \cL(\Omega_1)$ on $[k_dh,(k_d+1)h]$. We first  sum  over $k_d\in \cS_m$ for fixed $\t k_d$ and then over all $\t k_ d\in \cS_m^{d-1}$, 
which gives
\be
\label{zeroone}
\Sigma_1=\sum_{\t k_d\in \cS_m^{d-1}}
\,\,\int_{T_{\t k_d}}
\| f(\t x_d,\cdot)-I_h^1 f(\t x_d,\cdot)\|_{L_1(\Omega_1)}
d\t x_d
\leq 2h^2,
\ee
where we have used the result 
for $d=1$.

Next, consider the second sum $\Sigma_2$.
Each summand in $\Sigma_2$ is
\be\label{one}
\int_{T_k} E_2(x)dx=\(\int_{k_dh}^{(k_d+1)h}(1-\theta(x_d))dx_d\)
\(\int_{T_{\t k_d}} 
 |f(\t x_d,k_dh)-[ I^{d-1}_{h} f(\cdot,k_dh)](\t x_d)|d\t x_d\).
\ee
The first factor has value $\frac h 2$ and the second factor is the $L_1$ error of interpolation   of the $(d-1)$ dimensional function $f(\cdot,k_dh)$
on $T_{\t k_d}$. 
Similarly to what  we did in the estimate for $\Sigma_1$ and using \eref{one}, we have 
\begin{eqnarray}
\nonumber
\Sigma_2&:=&\sum_{k\in \cS_m^d}\int_{T_k} E_2(x)dx=\sum_{k_d\in \cS_m}\frac{h}{2}\sum_{\t k_d\in \cS_m^{d-1}}
\,\int_{T_{\t k_d}}|f(\t x_d,k_dh)-[I^{d-1}_{h} f(\cdot,k_dh)](\t x_d)|d\t x_d\\ \nonumber&=&\sum_{k_d\in \cS_m}\frac{h}{2}\|f(\cdot,k_dh)-I^{d-1}_hf(\cdot,k_dh)\|_{L_1(\Omega_{d-1})}.
\end{eqnarray}
Note that $f(\cdot,k_dh)$
belongs to $\cL(\Omega_{d-1})$. By the induction hypothesis, we have
\be\label{oneone}
\Sigma_2\leq \sum_{k_d\in\cS_m}\frac{h}{2}\cdot 2(d-1)h^2=(d-1)h^2.
\ee

The estimate for the third sum  $\Sigma_3$ in \eref{zero} is the same as \eref{oneone}, this time applying the inductive hypothesis for the $(d-1)$ dimensional function 
$f(\cdot,(k_d+1)h)$, namely,
\be
\label{two}
\Sigma_3:=\sum_{k\in\cS_m^d}\int_{T_k} E_3(x)dx\leq (d-1)h^2.
\ee
It follows from \eref{zero}, \eref{zeroone}, \eref{oneone}, and \eref{two} that 
\be
\|f-I_h^df\|_{L_1}\leq 2h^2+2(d-1)h^2=2dh^2,
\ee
which concludes the proof in the case $p=1$.
\nl
\nl
{\bf Case $p=\infty$:}  In this case, we only use the fact that the functions from $\cL$ are Lipschitz functions and do not utilize the fact that they are convex. We could therefore invoke classical results on multivariate spline approximation of functions such as \cite{dahmen1980multidimensional} where general estimates are stated in terms of moduli of smoothness, but we present an elementary proof for completeness.

For each $k\in \cS_m^d$, let $V_k$ denote the set of vertices of $T_k$.  Let us observe that the value $I^d_hf(x)$,   $x\in T_k$,  is a convex combination of the values of $f$ at these vertices, namely
\be
I_h^df(x)=\sum_{v\in V_k}\alpha_v(x)f(v), \quad \sum_{v\in V_k}\alpha_v(x)=1, \quad \alpha_v(x)\geq 0.
\ee
Since both $f$ and $I_h^df$ are continuous, there is $\bar x\in T_k$ for some index $k$, such that 
\begin{eqnarray}
\nonumber
\|f-I_h^df\|_{L_\infty}&=&
|f(\bar x)-I_hf(\bar x)|\leq 
\sum_{v\in V_k}\alpha_v(x)|f(\bar x)-f(v)|\\ \nonumber
&\leq& 
\sum_{v\in V_k}\alpha_v(x)|\bar x-v|
\leq \sqrt{d}h,
\end{eqnarray}
where we have used  
$|\bar x-v|\leq \sqrt{d}h$ whenever $v,\bar x\in T_k$.
We  can rewrite the estimate above as
\be
\|f-I_h^df\|_{L_\infty}
\leq  2dh,
\ee
so that we have the same constant as in the case $p=1$. The result for $1<p<\infty$ then follows from \eref{inter}.  This completes the proof of the Lemma.  
\end{proof}

\begin{remark} 
\label{R1}
Note that if we consider only Lipschitz functions with a Lipschitz constant $1$, we can obtain the upper bound $\|f-I_h^df\|_{L_\infty}\leq Ch$, with $C$ only depending on $d$, and   arrive at the known result 
$$
\|f-I_h^df\|_{L_p}
\leq  Ch.
$$
Observe that the gain of the extra factor $h^{1/p}$ in the error bound for functions from the class $\cL$ is due to the convexity of $f$, exploited in the $L_1$ estimate, where we gain an extra factor $h$.
\end{remark}

 Lemma \ref{L:Tub} serves to prove the upper bounds on tensor product sampling stated in Theorem \ref{T:tensorL}.  Namely, we have  proved \eref{ubL1}, from which \eref{sampl} follows, see \eref{lr}, \eref{serror}, and \eref{lerror}, for all values
$n=(m+1)^d$, and therefore for all $n\geq 1$,  according to Remark \ref{remdecrease}.

\subsection{Lower bounds for tensor product grids: proof of (ii) in Theorem \ref{T:tensorL}}
\label{SS:lb}

 Let $\Lambda:=\Gamma_1\times\cdots \times\Gamma_d$ be any tensor product grid
 in $\cT_n$.  Since $m_1\cdots m_d=n$, at least one of the sets $\Gamma_k$ consists of 
 $m:=m_k\le n^{1/d}$ points,  and therefore there is an interval $(a,b)\subset [0,1]$ of length 
 at least $h:=(m+1)^{-1}$  that does not contain any points from  $\Gamma_k$. Let $a^*:=(a+b)/2$ be the midpoint of this interval and  define  the functions $\pi,\tilde \pi:\Omega_d\to \R$, as follows
 $$
\pi(x):= \pi(x_1,\dots,x_d):=|x_k-a^*|,\quad {\rm and}\quad \tilde \pi(x):= \max\Big\{\frac h 2,\pi(x)\Big\}, \quad x=(x_1,\dots,x_d)\in\Omega.
 $$  
 Both of these functions are in $\cL$ and their values agree at all data sites.  We also have that the difference $E:=\tilde\pi-\pi$ satisfies
 \be 
 \label{havelb}
 E(x)\ge \frac h 4, \quad x\in\tilde\Omega:=\Big\{x\in\Omega:  |x_k-a^*|\le \frac h 4\Big\}.
 \ee
 Since $|\tilde \Omega|= \frac h 2$, this gives that
 \be 
 \label{lbE}
 \rho(\cL,\Lambda)_{L_p}\ge \frac{1}{2}\|E\|_{L_p}\ge \frac{1}{2}\(\frac h 2\)^{\frac 1 p}\frac h 4 \geq cm^{-(1+\frac 1 p)} \geq cn^{-r/d},
\ee
where the constant $c$ depends only on $d$. This proves the
lower bound \eref{orL1} in (ii).
 
\subsection {Lower bounds for  sampling numbers of $\cL$: proof of (iii), (iv), and (v) in Theorem \ref{T:tensorL}}
\label{proofiii}
 
 The upper bounds in (iii), (iv), and (v) of Theorem \ref{T:tensorL} follow from what we have already proved for tensor product grids. To prove the lower bounds, let $\Lambda$ be any set of $n$ sampling sites (not necessarily on a tensor grid). 
 
 We first consider the case $d=1$ and observe that there
 exists an interval $(a,b)\subset [0,1]$ of length at least $h:=(n+1)^{-1}$ 
 that does not contain any point from  $\Lambda$. We then proceed as
 in the proof of (ii), by introducing the functions $\pi$ and $\t \pi$
 (that are now univariate, i.e. $d=1$), and noting that \eref{lbE} becomes 
\be 
\label{lbEuniv}
\rho(\cL(\Omega_1),\Lambda)_{L_p}\ge   cn^{-r}.
\ee
Therefore, since $\Lambda$ is arbitrary,
\be 
\label{samploweruniv}
\rho_n(\cL(\Omega_1))_{L_p}\ge   cn^{-r},
\ee
for all $1\leq p\leq \infty$, which is the lower bound in \eref{orLgen} when $d=1$.

Next, we consider the  case $d> 1$ and  $p=\infty$. 
By standard packing arguments, there exists a point $x^*\in\Omega$ and a positive constant $c$ such that all the points in $\Lambda$ are outside the ball $B$
with center $x^*$ and radius $\delta= cn^{-1/d}$,  and this ball is completely contained in $\Omega_d$. The constant $c$ 
depends only on $d$. We then define  two functions $\phi,\tilde \phi:\Omega_d\to\R$ as follows
 \be
\phi(x):=\frac 1 {\sqrt d}|x-x^*|,\quad {\rm and}\quad \tilde \phi(x):= \frac 1 {\sqrt d}\max \{\delta,\phi(x)\}, \quad x\in\Omega_d.
\label{phis}
 \ee
Note that  $\phi,\tilde \phi\in\cL$ and their values are in agreement at all data sites.  Moreover, their difference satisfies
 \be 
 \label{havelb1}
 \|\tilde\phi-\phi\|_{L_\infty(\Omega_d)}\ge \tilde\phi(x^*)-\phi(x^*)=\frac{\delta}{\sqrt d}.
 \ee
 It follows that
 \be 
 \rho(\cL(\Omega_d),\Lambda)_{L_\infty}\ge \frac{\delta}{2\sqrt d} \geq cn^{-1/d},
\ee
where $c$ depends only on $d$. Since $\Lambda$ is  an arbitrary set of data sites, this proves the 
lower bound 
\be 
\label{samplowerLinf}
\rho_n(\cL(\Omega_d))_{L_\infty}\ge   cn^{-1/d},
\ee
which is the lower bound in \eref{orLgen} when $p=\infty$.

Next, we consider the case $d>1$ and $p=1$. In this case, we follow an approach already discussed in  \cite{KATSCHER19965}, where we chop off (while preserving the convexity) parts of a particular  function that are away from the given data sites. More precisely, we consider the function $\psi:\Omega_d\to\R$, defined as
\be
\psi(x):=\frac 1 {d+1} |x|^2=\frac 1 {d+1} (x_1^2+\cdots + x_d^2),
\label{f}
\ee
which belongs to $\cL$. Let $\Lambda$  be any set of $n$ sample sites in  $\Omega_d$. 
For an appropriate constant $c$ that depends only on $d$, we can always find
$2n$ points $y^1,\cdots,y^{2n}$ such that the balls $B(y^i,h)$  with center $y^i$ and radius 
$$
h=cn^{-1/d}
$$
are all disjoint and contained in 
$\Omega_d$. It follows that there exist $n$ points $y^i$ 
such that the balls $B(y^i,h)$ do not contain any points from $\Lambda
$.
Up to re-indexing we denote these points by $y^1,\dots,y^n$.

We will construct a function $\t \psi:\Omega\to\R$, $\t \psi\in \cL$, that differs from $\psi$ for  $x\in B(y^i,h)$ for $i=1,\dots,n$, and agrees with $\psi$ in the rest of $\Omega$, and thus at all data sites in $\Lambda$.  We notice that when $x$ is  on the boundary of $B(y^i,h)$ we have
$$
(d+1)\psi(x)=|x|^2=|y^i+x-y^i|^2=
(|x-y^i|^2 -|y^i|^2 + 2y^i\cdot x)=
(h^2-|y^i|^2 + 2y^i\cdot x),
$$
and therefore $\psi$ agrees on this boundary
with the affine function
$$
\psi_i(x)=a_i+\<b_i,x\>, \quad a_i=\frac 1{d+1}(h^2 -|y^i|^2),\quad
b_i= \frac 2{d+1}y^i.
$$
We define $\t \psi$ by
\be
\t \psi(x):=\begin{cases}
\psi_i(x),& \quad x\in B(y^i,h),\quad i=1,\dots,n,\\\\
\psi(x),&\quad x\in \Omega\setminus\cup_{i=1}^nB(y^i,h).
\end{cases}
\label{tf}
\ee
 By construction, $\t \psi$ is convex
and  belongs to $\cL$. For 
$x\in B(y^i,h)$, the difference between $\t \psi$ and $\psi$ is
$$
\t \psi(x)-\psi(x)=\psi_i(x)-\psi(x)=
\frac 1{d+1}(h^2-|x-y^i|^2) \geq 0.
$$
Thus, this difference is larger than $\frac{3}{4}\frac{h^2}{d+1}$ on the ball $B(y^i,h/2)$ for each $i=1,\dots,n$.   It follows that
\be
\label{L1error}
\|\t \psi-\psi\|_{L_1(\Omega_d)}=\int_\Omega (\t \psi(x)-\psi(x))dx \ge ch^2\ge cn^{-2/d},
\ee
where $c$ depends only on $d$. It follows that
\be
\label{lbp1}
\rho(\cL,\Lambda)_{L_1}\geq 
 c n^{-2/d}.
\ee 
Since $\Lambda$ is arbitrary, this shows that
\be 
\label{samplowerL1}
\rho_n(\cL)_{L_1}\ge cn^{-2/d},
\ee
which is the lower bound in \eref{orLgen} when $p=1$.

Next, we prove (v). We use the functions $\phi,\tilde \phi, \psi,\tilde\psi$ defined above. For each $1< p< \infty$, we 
know that
\be
\|\phi-\t \phi\|_{L_p} \geq c n^{-(\frac 1 d+\frac 1 p)},
\ee
because $|\phi -\tilde \phi| \ge cn^{-1/d}$ on a set of measure $cn^{-1}$.
We also know that 
 
\be
\|\psi-\t \psi\|_{L_p} \geq \|\psi-\t \psi\|_{L_1} \ge c n^{-2/d}.
\ee
This shows that a lower bound for the sampling numbers $\rho_n(\cL)_{L_p}$ is given 
by $c\max\{n^{-(\frac 1 d+\frac 1 p)},n^{-2/d}\}$, which is exactly \eref{lowLgen}.

Finally, we prove \eref{orL2} of (iv).
The upper bound is a consequence of (i). In order to obtain the  lower bound, we use the fact that $\rho_n^\ell(\cL)_{L_2}\geq d_n(\cL)_{L_2}$
and argue that the Kolmogorov widths, see \eqref{d:kol}, satisfy
\be
\label{kol2}
d_n(\cL)_{L_2} \asymp n^{-\frac{3}{2d}}.
\ee
We  only need to consider the case $d \geq 2$ since the one-dimensional result was proved in \cite{Konovalov2005}, see \eqref{gen}. In the case $d\geq 2$, for $\omega\in \mathbb{R}^d$ and $b\in \mathbb{R}$, we consider the ``hinge'' function:
\be
    \eta_{\omega,b}(x) := (\omega\cdot x+b)_+=\max\{\omega\cdot x+b,0\}, \quad x\in\Omega_d.
\ee
Note that for a dimension dependent constant $c$,  we have that whenever  $|\omega|,|b| \leq c$, the function $\eta_{\omega,b}\in \cL$. Hence, if we define
\be
    \cL_0 := \overline{\text{conv}\{\pm \,\eta_{\omega,b},~|\omega|,|b|\leq c\}}
\ee
to be the closure of the convex hull of $\pm \,\eta_{\omega,b}$, it follows that $\cL_0\subset \text{conv}\{-\cL\cup\cL\}$. On the other hand, by the triangle inequality,
one has
\be
d_n(\overline{\text{conv}\{-\cL\cup\cL\}})_{L_2}=d_n(\cL)_{L_2}.
\ee
Therefore, we have  $d_n(\cL)_{L_2} \geq d_n(\cL_0)_{L_2}$. We can now apply Theorem 9 from \cite{SX}, where it was shown that 
\be
d_n(\cL_0)_{L_2} \gtrsim n^{-\frac{3}{2d}}.
\ee
This completes the proof of  Theorem \ref{T:tensorL}.

\section{Optimal tensor product sampling for $\cB$}
\label{S:TensorB}

In this section, we prove the following result, again using the notation $r:=1+\frac 1 p$.

\begin{theorem}
\label{T:tensorB}
For the model class $\cB=\cB(\Omega_d)$,  we have the following results:
\vskip .1in
\noindent
{\rm (i)} For each $1\le p< \infty$, and $d\geq 1$, there is a non-uniform 
tensor product grid $\Lambda_n \in\cT_n$, depending on $p$, such that, 
\be
\label{ubL11}
\rho(\cB, \Lambda_n)_{L_p}\le  \rho^\ell(\cB, \Lambda_n)_{L_p}\le Cn^{-r/d},
\ee
with the constant $C$ depending only on $d$ and $p$. It follows that for $1\le p<\infty$ and every $n\geq 1$, 
\be
\label{sampb}
\rho_n(\cB)_{L_p}\le  \rho_n^\ell(\cB)_{L_p}\le Cn^{-r/d}.
\ee
\vskip .1in
\noindent
{\rm (ii)} For any tensor product grid $\Lambda \in\cT_n\subset \Omega_d$, $d\geq 1$,  any $1\le p< \infty$, and any $n\ge1$, we have the lower bounds
\be
\label{orL11}
\rho^\ell (\cB,\Lambda)_{L_p} \ge \rho (\cB,\Lambda)_{L_p}\ge c  n^{-r/d}, 
\ee  
where $c$ depends only on $d$.
\vskip .1in
\noindent
{\rm (iii)} In the case $d = 1$ and $1 \leq p< \infty$
we have
\be
\label{orBgen}
\rho_n(\cB)_{L_p}\asymp  \rho^\ell_n(\cB)_{L_p}\asymp n^{-r/d},
\ee
where the equivalency constants depend only on $d$ and $p$.
Also, the estimate \eqref{orBgen} holds if $d > 1$ and $p = 1$.

\vskip .1in
\noindent
{\rm (iv)} 
When  $p=2$ and  $d\geq 1$, one has
\be
\label{orL2B}
\rho^\ell_n(\cB)_{L_2}\asymp n^{-r/d}=n^{-\frac 3{2d}},
\ee
where the equivalency constants depend only on $d$.
\vskip .1in
\noindent
{\rm (v)}
When  $1<p< \infty$ and $d>1$, one has 
\be
\label{lowBgen}
\rho_n(\cB)_{L_p}\geq c n^{-\t r/d}, \quad \t r:=\min\{2,1+d/p\}<r,
\ee
where $c$ depends only on $d$.
\end{theorem}

One main difference between Theorem \ref{T:tensorB}
and Theorem \ref{T:tensorL} is that non uniform tensor product
grids are required. More precisely, we need sampling grids
that get finer near the boundary of $\Omega$.  The reason for this is that we no longer have a uniform bound on the subgradient of  functions  $f\in\cB$. However, because of convexity, we know that,  the subgradients can only get large 
near the boundary.  Another difference is that the non uniform  tensor product grid depends on the value $p$ of the $L_p$ norm in which we measure error.  Otherwise, the proof given below follows the lines
of the proof for $\cL$.

Note  that the value $p=\infty$ is excluded 
(whereas all rates of decay are the same for both classes 
$\cL$ and $\cB$ when $p<\infty$). In fact, it is easily seen that
when $p=\infty$, the sampling numbers do not decay towards $0$,
that is,
\be
\rho_n(\cB)_{L_\infty}\asymp  \rho^\ell_n(\cB)_{L_\infty}\asymp 1.
\ee
Consider, for example, the  case $d=1$ and observe that 
for any data sites $0\leq x^1 <\cdots< x^n \leq 1$, we can
create two functions in $\cB$ that agree at these points and 
differ by $1/2$ in the $L_\infty$ norm.   If $x^1>0$,
we take the functions $\(1-\frac {x}{x^1}\)_+$ and $\(1-2\frac {x}{x^1}\)_+$, and in the case $x^1=0$,
we replace $x^1$ by $x^2$ in the above definitions. This example can be easily modified to create functions with the same property in any dimension $d>1$.

Observe that since $\cL\subset \cB$, we only need to prove part (i), since the other statements in (ii), (iii) and (iv) will follow from the lower bounds that were proved for $\cL$ in the previous section. Therefore, we limit ourselves to the proof of (i) as we proceed forward in this section.

\subsection{Upper bounds for tensor product grids: proof of Theorem \ref{T:tensorB}, (i),   $d=1$}
\label{SS:tensorB1}

We first prove the upper bound
\eref{ubL11} when  $d=1$.
We begin by noting that 
for every univariate function 
$g\in\cB$, any  $t\in (0,1)$,   one
has the upper and lower bounds 
\be
\label{uplow1}
-\frac {2} t\leq  g_-'(t) \le g_+'(t)\leq \frac {2}{1-t}.
\ee
Indeed, since $\|g\|_{L_\infty([0,1])}\le 1$, these bounds follow from the definition \eref{subgradient} of the subgradient, applied to $y=0$, $x=t$,  and
$y=1$, $x=t$, respectively.

We fix the value of $p\in [1,\infty)$ and let $m\in \mathbb N$. We define the   sampling sites $t_k$ as follows:
\be 
\label{tk}
t_k:=\frac{1}{2}\left (\frac{k}{m}\right)^{p+1}, \quad k=0, \ldots,m,
\ee
with $t_0=0$, $t_m=\frac{1}{2}$, 
and 
\be
\label{tkrest}
t_{2m-k}:=1-t_k,\quad k=0, \ldots,m-1,
\ee
with $t_{2m}=1$, which are thus symmetric with respect to $\frac{1}{2}$.   Thus,  the set of data sites is $\Lambda_n:=\{t_0<t_1<\cdots<t_{2m}\}$ and their number is $n=2m+1$.

We consider the interpolation operator $I_n:=I_n^1:=I_{\Lambda_n}$, which assigns to any $f\in\cB(\Omega_1)$ the continuous piecewise linear interpolant of
$f$ at the $n$ points $t_0,\dots,t_{2m}$.  
In order to prove \iref{ubL11}, we will show that 
for any $1\le p<\infty$ and $f\in\cB(\Omega_1)$, we have
\be 
\label{ubp1}
\|f-I_nf\|_{L_p(\Omega_1)}\le Cn^{-r},
\ee
where $C$ depends only on $p$.

We begin by noting that for $T_k:=[t_k,t_{k+1}]$, $k=1, \ldots,m-1$, we have
\be
|T_k|=t_{k+1}-t_k\asymp k^pm^{-p-1},
\ee
with constants in this equivalence and the inequalities that follow in this proof depending only on $p$.
The local estimate of
Lemma \ref{L:intlemma} applied to $T_k$ gives for $k=1,\dots,m-1$,
\begin{eqnarray}
\label{eq11}
\|f-I_nf\|^p_{L_p(T_k)}
\leq   |T_k|^{p+1}(f'_-(t_{k+1})-f'_+(t_k))^p\lesssim  m^{-(p+1)^2}k^{p(p+1)}(f'_-(t_{k+1})-f'_+(t_k))^p,
\end{eqnarray}
where the multiplicative constants in $\lesssim$  here and later in this section only depend on $p$.
For the extreme left interval $T_0=[0,\frac 1 2 m^{-p-1}]$  we have  
\be
\|f-I_nf\|^p_{L_p(T_0)}\leq |T_0|\|f-I_nf\|^p_{L_\infty(T_0)}\leq 2^{p-1}  m^{-p-1},
\ee
because both $f$ and $\cI_nf$ have $L_\infty$ norms at most one.

We next focus on  $[0,1/2]$  and use the above estimates, together with \eref{eq11}, to obtain
\begin{eqnarray}
\label{eqnarray1}
\nonumber
\|f-I_nf\|^p_{L_p([0,1/2])}&=&\sum_{k=0}^{m-1}
\|f-I_nf\|^p_{L_p(T_k)}
\\ \nonumber
&\lesssim& 
m^{-p-1}+m^{-(p+1)^2}
\sum_{k=1}^{m-1} k^{p(p+1)}[f'_-(t_{k+1})-f'_+(t_k)]^p\\ \label{half1pp}
&\leq&
m^{-p-1}+m^{-(p+1)^2}
\sum_{k=1}^{m-1} \left [k^{(p+1)}(f'_+(t_{k+1})-f'_+(t_k))\right ]^p,
\nonumber
\end{eqnarray}
where
we have used that  $f'_-(t_{k+1})\leq f'_+(t_{k+1})$, due to the convexity of $f$.
We bound the sum in the second term by
\be
\sum_{k=1}^{m-1} \left [k^{p+1}(f'_+(t_{k+1})-f'_+(t_k))\right]^p\leq \left (\sum_{k=1}^{m-1} k^{p+1}[f'_+(t_{k+1})-f'_+(t_k)]\right )^p=:S^p,
\ee
which  gives
\be
\label{main1}
\|f-I_nf\|^p_{L_p([0,1/2])}\lesssim m^{-p-1}+m^{-(p+1)^2} S^p.
\ee

To bound $S$, we perform   summation by parts.   Let $K< m $ be the largest index $k$ for which $f_+'(t_k)<0$.  Summing by parts, we obtain
\be 
\label{boundS}
S\lesssim (m-1)^{p+1}f'_+\left(t_m\right)
-f'_+\left(t_1\right)-\sum_{k=2}^{m-1}f'_+(t_k)[k^{p+1}-(k-1)^{p+1}].
\ee
  We use the estimate  \iref{uplow1} to obtain the bounds
\be
f_+'(t_m)=f'_+\left(\frac{1}{2}\right)\leq 4,
\ee
and  
\be
-f_+'(t_1)=-f'_+\left(\frac{1}{2}m^{-p-1}\right)\leq 4m^{p+1}.
\ee
Therefore, 
\be
\label{boundS1}
S\le 8m^{p+1}+\Sigma,
\ee
where
\be 
\label{Sigma}
\Sigma:= \sum_{k=2}^{K}|f'_+(t_k)|[k^{p+1}-(k-1)^{p+1}].
\ee
Since 
$k^{p+1}-(k-1)^{p+1}\asymp k^p$, 
we have

\be 
\label{Sigma1}
\Sigma
\lesssim \sum_{k=2}^K|f'_+(t_k)|k^{p}.
  \ee
 By convexity, we note that
for $k\in [2,K]$, one has
\be
k^{p}m^{-p-1} |f'_+(t_k)|\asymp|T_{k-1}||f'_+(t_k)|\leq f(t_{k-1})-f(t_k),
\ee
and therefore 
 $\Sigma \lesssim m^{p+1}$ because $\|f\|_{L_\infty(\Omega_1)}\le 1$.

It follows then from \eref{boundS1} that 
\be
S\lesssim m^{p+1},
\ee
and substituting this into \eref{main1} gives
\be
\|f-I_nf\|^p_{L_p([0,1/2])}\lesssim m^{-p-1}.
\ee
The same inequality  holds for $\|f-I_nf\|^p_{L_p([1/2,1])}$, and since $n\asymp m$, we reach \eref{ubp1} when $d=1$.

\subsection{Upper bounds for tensor product grids: proof of  Theorem \ref{T:tensorB}, (i),  $d>1$}

The proof for $d> 1$ is similar to the proof 
of Theorem \ref{T:tensorL} for the model class $\cL$ since it uses  induction on the dimension $d$ and a multilinear interpolation operator, based on a tensor product grid. Let $\Gamma:=\{t_0<t_1<\cdots<t_{2m}\}$, $m\in \mathbb N$,  be the set of points 
that were introduced in the case $d=1$, see \eqref{tk}-\eqref{tkrest}.  We define our tensor product grid for $[0,1]^d$ as $\Lambda_n:=\Gamma\times\cdots\times \Gamma$.  Notice that the spacing of the grid points is the same in each coordinate direction.  In other words, the data sites in $\Lambda_n$ are
\be
x^k=(t_{k_1},\dots,t_{k_d}), \quad k=(k_1,\dots,k_d)\in \cS_{2m}^d,
\ee
where $\cS_{2m}=\{0,\dots,2m-1\}$.
The number of data sites is now
\be
n=(2m+1)^d.
\label{htn}
\ee
 We denote by $I_n^d:=I_{\Lambda_n}$ the interpolation operator given  in \S\ref{SS:pwlinterpolation}.  
We denote by $T_k$ the rectangular cell
\be
T_k:=[t_{k_1},t_{k_1+1}]\times\cdots\times [t_{k_d},t_{k_d+1}], \quad k=(k_1,\dots,k_d)\in \cS_{2m}^d,
\ee
 and if $k=(\t k_d,k_d)$, where 
$\t k_d:=(k_1,\dots,k_{d-1})\in \{0,\dots,2m-1\}^{d-1}$,
we have 
\be
\quad T_k=T_{\t k_d}\times [t_{k_d},t_{k_d+1}].
\ee
Given $f\in\cB(\Omega_d)$, $I_n^df$ can be 
described  inductively for  $x=(\tilde x_d,x_d)\in T_k$ by
\be
\label{inductp}
I^d_n f(x_1, \ldots,x_d)=(1-\theta(x_d)) [I^{d-1}_{n'} f(\cdot,t_{k_d})](\t x_d)+\theta(x_d)
[I^{d-1}_{n'} f(\cdot,t_{k_d+1})](\t x_d), 
\ee
where
\be
n:=(2m+1)^d, \quad n':=(2m+1)^{d-1}, \quad \theta(x_d):=\frac{x_d-t_{k_d}}{t_{k_d+1}-t_{k_d}}\in [0,1], \quad x_d\in [t_{k_d}, t_{k_d+1}].
\ee

Before going further, let us observe that in the case of  $p=\infty$, 
we have
\be
\|f-I^d_nf\|_{L_\infty}\leq 2,
\ee
since $I^d_n$ does not increase the 
$L_\infty$ norm. As already explained, we cannot expect any decay of the error in the $L_\infty$ norm.

In order to prove \iref{ubL11}, we will show that 
for any $1\le p<\infty$ and $f\in\cB(\Omega_d)$, $d\geq 1$, we have
\be
\label{toprove1}
\|f-I^d_nf\|_{L_p}\leq C n^{-r/d},\quad n\ge 1,
\ee
where $C$ depends only on $d$ and $p$. 

To prove \eref{toprove1},
we use induction on $d$.  We have already proved the validity of \eref{toprove1} for $d=1$. Assume
that \iref{toprove1} holds for dimension $(d-1)$.   
We note that for $k=(\tilde k_d,k_d)\in \cS_{2m}^d$
and  $x=(\t x_d,x_d)\in  T_k=T_{\t k_d}\times [t_{k_d}, t_{k_d+1}]$, we have that 
$$
\begin{array}{ll}
\|f(x)-I^d_n f(x)\|_{L_p(T_k)}
& = \|f(x)-(1-\theta(x_d)) [I^{d-1}_{n'} f(\cdot,t_{k_d})](\t x_d)-\theta(x_d)
[I^{d-1}_{n'} f(\cdot,t_{k_d+1})](\t x_d)\|_{L_p(T_k)} \\\\
& \leq  \|E_1(x)\|_{L_p(T_k)}+\| E_2(x)\|_{L_p(T_k)}+\| E_3(x)\|_{L_p(T_k)},
\end{array}
$$
where we have used the notation,
\be
E_1(x):=f(\t x_d,x_d)-(1-\theta(x_d)) f(\t x_d,t_{k_d}) -\theta(x_d) f(\t x_d,t_{k_d+1}),
\ee
\be
E_2(x):=(1-\theta(x_d))(f(\t x_d,t_{k_d})-[I^{d-1}_{n'} f(\cdot,t_{k_d})](\t x_d)),
\ee
and
\be
 E_3(x):=\theta(x_d) (f(\t x_d,t_{k_d+1})-[ I^{d-1}_{n'} f(\cdot,t_{k_d+1})](\t x_d)).
\ee
We thus have
\begin{eqnarray}
\nonumber
\|f-I_nf\|^p_{L_p(\Omega_d)}&=&
\sum_{k\in\cS_{2m}^d} 
\|f-I_nf\|^p_{L_p(T_k)}
\leq 
\sum_{k\in\cM_d} 
\left [\sum_{i=1}^3\|E_i(x)\|_{L_p(T_k)}\right]^p
\\ \nonumber
&\leq&
4^{p-1}\left (\sum_{k\in\cM_d} 
\|E_1(x)\|^p_{L_p(T_k)}+\sum_{k\in\cM_d} \| E_2(x)\|^p_{L_p(T_k)}+\sum_{k\in\cM_d} \|E_3(x)\|^p_{L_p(T_k)}\right)\\
\label{zerop}
&=:& 4^{p-1}(\Sigma_1+\Sigma_2+\Sigma_3).
\end{eqnarray}
.

For the first sum, we write
\be
\Sigma_1:=\sum_{k\in\cS_{2m}^d}\int_{T_k} |
E_1(x)|^p\,dx=\sum_{\t k_d\in\cS_{2m}^{d-1}}\int_{ T_{\t k_d} }\left [\int_{t_{k_d}}^{t_{k_d+1}}
 |E_1(\t x_d,x_d)|^p\,dx_d\right ] d\t x_d,
\ee
and observe that for each fixed $\t x_d\in T_{\t k_d}$, the univariate integral is the $p$-th power of the $L_p$  error  on $[t_{k_d},t_{k_d+1}]$ between the univariate function $g:=f(\t x_d,\cdot)$ and its interpolant $I_m^1g$.  Notice that $g$ is in $\cB(\Omega_1)$. To evaluate $\Sigma_1$, we  sum  over $k_d$ for fixed $\t k_d$ and then over all $\t k_d\in \cM_{d-1}$. Using the univariate  interpolation error estimate for $g$, see \eref{ubp1}, we obtain
\begin{eqnarray}
\nonumber
\Sigma_1&=&\sum_{\t k_d\in\cS_{2m}^{d-1}}
\,\int_{T_{\t k_d}}\sum_{k_d=0}^{2m-1} \left [\int_{t_{k_d}}^{t_{k_d+1}}
 |E_1(\t x_d,x_d)|^p\,dx_d\right ] d\t x_d=
 \sum_{\t k_d\in\cS_{2m}^{d-1}}
\,\int_{T_{\t k_d}}\left [\int_0^1|E_1(\t x_d,x_d)|^p\,dx_d\right ]d\t x_d\\ \label{zeroonep}
&\lesssim& m^{-p-1} \sum_{\t k_d\in\cS_{2m}^{d-1}}
\,\int_{T_{\t k_d}}1\,d\t x_d
=m^{-p-1}\lesssim n^{-rp/d} ,
\nonumber
\end{eqnarray}
where we have used that $m\asymp n^{1/d}$.

Each summand in the second sum $\Sigma_2$ is
\begin{eqnarray}
\nonumber
\int_{T_k} | E_2(x)|^p\,dx&=&\(\int_{t_{k_d}}^{t_{k_d+1}}(1-\theta(x_d))^pdx_d\)
\(\int_{T_{\t k_d}} 
 \|f(\cdot,t_{k_d})-[ I^{d-1}_{n'} f(\cdot,t_{k_d})](\t x_d)\|^pd\t x_d\)\\ \label{oneonep}
 &\lesssim&
 m^{-1}\|f(\cdot,t_{k_d})-[ I^{d-1}_{n'} f(\cdot,t_{k_d})]\|_{L_p(T_{\tilde k_d})}^p,
\end{eqnarray}
where we have used that $\int_{t_{k_d}}^{t_{k_d+1}}(1-\theta(x_d))^pdx_d\lesssim |t_{k_d+1}-t_{k_d}|\leq m^{-1}$.
Similarly to what  we did in the estimate for $\Sigma_1$ and using \eref{oneonep}, we have 
\begin{eqnarray}
\nonumber
\Sigma_2&:=&\sum_{k\in\cS_{2m}^d}\int_{T_k} |E_2(x)|^pdx\lesssim\sum_{k_d\in\cS_{2m}}m^{-1}\sum_{\t k_d\in\cS_{2m}^{d-1}}
\|f(\cdot,t_{k_d})-[ I^{d-1}_{n'} f(\cdot,t_{k_d})]\|_{L_p(T_{\tilde k_d})}^p\\ \nonumber
&=&\sum_{k_d\in\cS_{2m}}m^{-1}\|f(\cdot,t_{k_d})-I^{d-1}_{n'}(f\cdot,t_{k_d})\|^p_{L_p(\Omega_{d-1})}
\lesssim
\sum_{k_d\in\cS_{2m}}m^{-1}[n']^{-\frac{rp}{d-1}}\lesssim n^{-rp/d},
\nonumber
\end{eqnarray}
where we have used the induction hypothesis for the $(d-1)$ dimensional convex function $f(\cdot,t_{k_d})$. 

The estimate for the third sum $\Sigma_3$ in \eref{zerop} is handled the same way as the estimate for $\Sigma_2$, this time applying the inductive hypothesis for the $(d-1)$ dimensional function 
$f(\cdot,t_{k_d+1})$, and obtaining that 
\be
\label{twop}
\Sigma_3\lesssim  n^{-rp/d}.
\ee
Using  the estimates for $\Sigma_1,\Sigma_2,\Sigma_3$ in \eref{zerop} gives \iref{toprove1}.

We have therefore proved \iref{ubL11}.
The estimate \eref{sampb} follows for all values
$n=(2m+1)^d$, and, on account of Remark \ref{remdecrease}, therefore for all $n\geq 1$.
This completes the proof of Theorem \ref{T:tensorB}. 

\section{Sharper upper bounds for the sampling numbers of $\cL$}
\label{S:optimal}

It turns out that sampling on tensor product grids, as analyzed in the previous two sections, does not give the optimal sampling rates for the model classes $\cL$ and $\cB$.  In this section, we prove the following  new upper bounds for the sampling rates
of $\cL$.

\begin{theorem}\label{T:further-improved}
    Let $d \geq 2$ and $1 < p < \infty$. Then,  for any $n \geq 2$, we have
    \begin{equation}\label{improved-upper-bound-equation}
    \rho_n(\cL)_{L_p} \leq C \kappa(n,p,d):=  C\begin{cases}
                        n^{-(\frac 1p+\frac 1d)}(\log n)^{\frac 2 p}, & p > d,\\
              n^{-\frac 2 d}(\log{n})^{\frac 3d}, & p = d,\\
             n^{-\frac 2 d}(\log{n})^{\frac 2d}, & p < d,
        \end{cases}
\end{equation}
where $C$ depends only on $p$ and $d$\footnote[2]{Throughout this paper $\log$ indicates the logarithm to the base $2$.}. 
\end{theorem}

The bounds \eref{improved-upper-bound-equation} are significantly better than those obtained using tensor product grids.
Moreover, these new bounds are optimal up to the logarithmic factors, since  we have
proved in (iv) of Theorem \ref{T:tensorL} that there is a $c>0$ depending only on $d$ and $p$ such that for $\t r=\min\{2,1+d/p\}$  we have
\be
\rho_n(\cL)_{L_p} \geq cn^{-\t r/d}, \quad  n\ge 1.
\ee 
Recall that the upper bound $Cn^{-\t r/d} $ for this sampling rate is attained without a logarithmic factor when $p=1$ and $p=\infty$, see (iii) of Theorem \ref{T:tensorL}.
\nl

We fix $d\in\{1,2,\dots\}$ for the remainder of this section.  All constants that appear below in this section will depend only on $d$. In order to establish Theorem \ref{T:further-improved}, we will use  particular sets of $n$ 
non tensor product data sites, described in the next subsection.

\subsection{Data sites  used for sharper upper bounds for the sampling numbers of $\cL$}\label{ssec:datasites}

In this subsection, we  prove for every $n\geq 2$ the existence of special data sites $\Lambda^*_n:=\{x^1,\dots,x^n\}\subset \Omega$,
having the following two properties: 
\vskip .1in
\noindent
{\bf Property 1:}  Any closed convex set $S\subset \Omega$ with no data site $x^i$ in its interior ${\rm int}(S)$  satisfies the following estimate for its
 volume $|S|$,
 \be
|S|\leq C\frac {\log n}{n}, \quad n\geq 2,
\label{Ssize}
\ee
for some constant $C$.

\vskip .1in
\noindent
{\bf Property 2:} For all $x\in \Omega$, one has
\be
\min_{i=1,\dots,n} |x-x^i| \leq Cn^{-1/d}, \quad n\geq 2.
\label{cover}
\ee
Both constants $C$  depend only on $d$.

We shall use probabilistic methods to prove the existence of sets $\Lambda_n^*$ that satisfy these two properties. %
Consider the tensor product grid $\Lambda_m$ with spacing $1/m$, where $m$ is the largest integer such that $m^d \leq n$. We continue to use our previous notation for such tensor product grids, in particular the notation $T_k$ for the cells, $k\in \cS_m^d$, produced by the grid (see \eqref{t-k-definition}). For each $k\in \cS_m^d$ we choose a point $x_k$ uniformly at random from the cell $T_k$.  The number of these randomly chosen points is 
\be
\label{pp}
c'n< m^d\le n, \quad c'=c'(d).
\ee
We additionally choose $n-m^d$ points chosen from $\Omega_d$ at random  and let $\Lambda_n^*$ denote the corresponding (random) set of points. It is clear that $\Lambda_n^*$ automatically satisfies Property 2, since for each $k\in \cS_m^d$ the cell $T_k$ contains a point from $\Lambda_n^*$. The following result shows that with positive probability it also satisfies Property 1.
\begin{prop}\label{property-1-proposition}
    For each $d=1,2,\dots$, there are constants $L,A,\alpha>0$,  depending only upon $d$, such that for every $n \geq 2$ and every   $C\ge 1$, we have
    \begin{equation}
    \label{toprovec}
   \pi(C;n,d):=  \    \mathbb{P}\left(\exists S\subset\Omega:~ \text{$S$ is convex, } |S| \geq C\frac{\log n}{n},~\text{and $S\cap \Lambda_n^* = \emptyset$}\right) < Ln^{A-\alpha C}.
    \end{equation}
    In particular, the probability that $\Lambda_n^*$ doesn't satisfy Property 1 is less than $1$ if $C$ is sufficiently large.
\end{prop}

The first step in the proof of Proposition \ref{property-1-proposition} is to reduce the set of convex sets $S\subset \Omega$ to a finite set of simplices. To do this, let $r \geq 1$ be an integer (which will be determined later) and consider the uniform grid $\Lambda_r$ with spacing $1/r$. We have the following simple geometric result, which essentially bounds the measure of the boundary of a convex set.
\begin{lemma}
\label{convex-error-lemma}

    Let $r\ge 1$ be an integer and let $S\subset \Omega_d $ be closed convex set. Then, we have
    \begin{equation}\label{convex-error-equation}
        |S  \setminus \text{\normalfont conv}(S\cap\Lambda_r)| \leq 2d^2 r^{-1}.
    \end{equation}
    \end{lemma}
\begin{proof}
    Note that the difference $S \setminus \text{\normalfont conv}(S\cap\Lambda_r)$ is  covered by the set of cubes $T_k$ for $k\in \cS_r^d$ which intersect both $S$ and $S^c$. It is known that the number of such cubes is bounded by $Cr^{d-1}$ (see for instance \cite{lassak1988covering}), which completes the proof since the measure of each such cube is $r^{-d}$. However, for the reader's convenience, we give a simple proof of this bound.  

    Given any cube $T_k$, $k\in \cS_r^d$, with  $T_k\cap S\neq \emptyset$, and a coordinate direction $v\in \{\pm e_i:~i=1,...,d\}$, we define the height $h(T_k,v)$ of $T_k$ in the direction $v$ as follows:
    \be
        h(T_k,v) := \max\{t\in \mathbb{Z}_{\geq 0}:~T_{k+tv}\cap S \neq \emptyset\}.
    \ee
    In words, the height $h(T_k,v)$ is the maximum number of cubes one can move in the $v$-direction which still intersect $S$.
    
    We claim that if $T_k$ intersects both $S$ and $S^c$, then there exists a direction $v$ such that 
    \be
    \label{claim2}
    h(T_k,v) \leq d-1.
    \ee
   Indeed, suppose that $T_k$ intersects both $S$ and $S^c$. Let $y\in T_k \cap S^c$ and take a hyperplane $H$ separating $y$ from $S$, and let $w$ be the unit normal to $H$ such that   $w\cdot x < w\cdot y$ for all $x\in S$. Let $v^*$ be a coordinate direction such that 
    \be
    w\cdot v^* \geq \frac{1}{\sqrt{d}}.
    \ee
    Such a direction can always be found by taking the  coordinate of $w$  with the largest absolute value.

Now, suppose that $z\in T_{k+tv^*}$. Observe that $z - y = tv^* + \rho$ where $|\rho| \leq \sqrt{d}r^{-1}$. This implies that
\be
    w\cdot z = w\cdot y + tw\cdot v^*+ w\cdot \rho \geq w\cdot y + \frac{t}{\sqrt{d}} - \sqrt{d}.
\ee
Now, if $t \geq d$, this implies that $w\cdot z \geq w\cdot y$, and thus $z\notin S$. Hence $T_{k+tv^*} \cap S = \emptyset$ for all $t \geq d$, which gives  that $h(T_k,v^*) \leq d-1$ as desired. This means that
\begin{equation}\label{e:union_bound}
    \{T_k:~T_k\cap S \neq \emptyset~\text{and}~T_k\cap S^c\neq \emptyset\} \subset \bigcup_{v} \{T_k:~h(T_k,v) \leq d-1\}.
\end{equation}
We then observe that for each direction $v$ we have
    \be
        \# \{k\in \cS_r^d:~h(T_k,v) \leq d-1\} \leq dr^{d-1},
    \ee
    which in view of \eqref{e:union_bound} completes the proof.     

\end{proof}

We will also need the following straightforward geometric result, which will allow us to reduce the proof of Proposition \ref {property-1-proposition} to the case that $S$ is  a simplex.

\begin{lemma}\label{simplex-ontainment-lemma-polyhedron}
    Every closed convex set $S\subset \Omega_d$ contains a closed simplex $\Delta$ such that  
    $$
    |\Delta| \geq c|S|, \quad  c=c(d)=\frac{2^{1-d}}{d!}.
    $$
    Moreover, the vertices of $\Delta$ can be taken as extreme points of $S$.
\end{lemma}

\begin{proof} The lemma follows from a closely related result of Hadwiger  \cite{hadwiger1955volumschatzung}.  For the reader's convenience, we give the following proof  which constructs the  vertices of $\Delta$ using a greedy-type procedure.
We start with the following two points from $S$. 
    Let $z^0,z^1\in S$ be points of maximal distance, i.e., such that
    \begin{equation}
        |z^1 - z^0| = \max_{x,y\in S} |x - y|.
    \end{equation}
    Observe that the function $f(z^0,z^1) = |z^1 - z^0|$ is a convex function on the set $S\times S$. Hence, it achieves its maximum at an extreme point of $S\times S$. Thus, we may take $z^0$ and $z^1$ to be extreme points of $S$.
    
    Let $v_1$ be the unit vector  $v_1 :=\frac 1 {|z^1 - z^0|} (z^1 - z^0)$. Note  that $S$ is contained in the strip
    \begin{equation}
        L_1 := \{x\in \mathbb{R}^d: z^0\cdot v_1 \leq x\cdot v_1 \leq z^1\cdot v_1\}
    \end{equation}
    of width $|z^1 - z^0|$. Indeed, if $x\in S$ and  $x\cdot v_1 > z^1\cdot v_1$, then 
    \be
    |x - z^0| \geq (x - z^0) \cdot v_1 > (z^1 - z^0)\cdot v_1 = |z^1 - z^0|,
    \ee
    which contradicts the maximality of $|z^1 - z^0|$. A similar argument holds if we assume that $x\in S$ and  $x\cdot v_1 < z^0\cdot v_1$.

  Next, we inductively construct points $z^k\in S$ for $k=2,...,d$,  as follows. Suppose that $z^0,...,z^{k-1}$, have 
  been chosen. We denote by $F_k$  the $(k-1)$-dimensional affine space spanned by 
  these points, i.e.,
    \begin{equation}
        F_k := \left\{\sum_{i=0}^{k-1}a_iz^i \; :\; \sum_{i=0}^{k-1}a_i = 1\right\}.
    \end{equation}
    We find a point  $z^k\in S$ at maximal distance to $F_k$, i.e.,
    \begin{equation}
        \dist(z^k,F_k) = \max_{x\in S} \dist(x,F_k).
    \end{equation}
    Note that the distance to $F_k$ is a convex function, and thus $z^k$ may be taken to be an extreme point of $S$.
    
    Denote by $P_kz^k:= \arg\min_{y\in F_k}|y - z^k|$ its projection onto $F_k$, and let $v_k$ be the unit vector defined as
    $$
    v_k := \frac 1 {\dist(z^k,F_k)}(z^k - P_kz^k)=\frac 1 {|z^k-P_kz^k|}(z^k - P_kz^k).
    $$
Note that $v_k$ is orthogonal to the affine space $F_k$, and thus, by construction,  
 all  vectors $v_1,\dots,v_k$ form an orthonormal system. We observe that $S$ is contained in the strip
    \begin{equation}
        L_k := \{x\in \mathbb{R}^d: (2P_kz^k- z^k)\cdot v_k\leq x\cdot v_k \leq z^k\cdot v_k\}
    \end{equation}
    of width $2|z^k-P_kz^k|$. This holds since if $x\in S$ but $x\notin L_k$, then 
    $$|x - P_kz^k|\geq |(x - P_kz^k)\cdot v_k| > (z^k - P_kz^k)\cdot v_k=|z^k - P_kz^k|,
    $$
 which implies $\dist(x,F_k) > \dist(z^k,F_k)$, and thus contradicts the choice of $z^k$. Similarly, one shows that 
    $x\in S$ and $x\cdot v_k>z^k\cdot v_k$ is not possible.

    Once we have constructed the points $z^0,...,z^d$, we consider the simplex 
    $$
    \Delta := \text{\normalfont conv}(z^0,...,z^d).
    $$
    Clearly $\Delta\subset S$ since $S$ is convex and each $z^i\in S$. Recall from the construction that each $z^i$ is also an extreme point of $S$. 
    
    Using that the directions $v_k$ are orthogonal,
 we calculate that the volume $|\Delta|$ is
    \begin{equation}
        |\Delta| = 
     \frac{1}{d
        !}|z^1-z^0|\prod_{k=2}^d|z^k-P_kz^k|.
    \end{equation}
Since $S$ is contained in the rectangle $R= \bigcap_{k=1}^dL_k$, its Lebesgue measure satisfies
    \begin{equation}
        |S| \leq |R|= |z^1-z^0|\prod_{k=2}^d(2|z^k-P_kz^k|) = 2^{d-1}d!|\Delta|,
    \end{equation}
    which is the bound we wanted to show.
    \end{proof}

    We will now give the proof of Proposition \ref{property-1-proposition}.
    
    \begin{proof}[{\bf Proof of Proposition \ref{property-1-proposition}}]
   Let $n\ge 2$ and $d\ge 1$ be arbitrary but fixed. Given any $C\geq 1$, we consider the collection
    \be 
    \label{csets}  \Gamma:=\Gamma(C;n,d):=\left\{S: \ S \ {\rm is \ closed, \ convex \ and }\ |S| > C\frac{\log{n}}{n}\right\}.
    \ee
    
    Let us choose $r=4d^2n$, which will make the right hand side of \eqref{convex-error-equation}   less than $\frac{1}{2}\frac{\log{n}}{n}$. Then for $S\in \Gamma$, using Lemma \ref{convex-error-lemma}, we will have that 
    \be
        |\text{\normalfont conv}(S\cap\Lambda_r)| \geq \frac{C}{2}\frac{\log{n}}{n}.
    \ee
    Further, by applying Lemma \ref{simplex-ontainment-lemma-polyhedron} to the set $\text{\normalfont conv}(S\cap \Lambda_r)$ (whose extreme points clearly lie in $\Lambda_r)$, we obtain a simplex $\Delta$ with vertices lying on $\Lambda_r$ such that $\Delta \subset S\cap\Lambda_r \subset S$, and 
    \be
        |\Delta| \geq c\frac{C}{2}\frac{\log{n}}{n},
    \ee
    where $c=c(d)$ is the constant in Lemma~\ref{simplex-ontainment-lemma-polyhedron}.  In going further, we define $\tilde \Gamma:=\tilde \Gamma(C;n,d)$ to be the set of all such simplices.  
    
    We know that for any  $S\in\Gamma$, there is a $\Delta\in \tilde \Gamma$ such that $\Delta\subset S$.     Therefore,    we have the estimate
    
    \be\label{desired-probability-bound-eq}
       \pi(C;n,d) \le  \mathbb{P}\left(\exists \ \Delta\in \tilde \Gamma \ \text{with $\Delta\cap \Lambda_n^* = \emptyset$}\right)=:\tilde\pi(C;n,d).
    \ee
    To bound $\tilde\pi$, we consider any  simplex $\Delta\in \tilde\Gamma$  and the probability
    \be
        \mathbb{P}(\Delta \cap \Lambda_n^* = \emptyset).
    \ee
   Recall that for  the regular grid $\Lambda_m$,  $m$ is the largest integer such that $m^d\le n$.  We fix a cell $T_k$ with $k\in \cS_m^d$. The point in $\Lambda_n^*$ which is chosen randomly from $T_k$ does not lie in $\Delta$ with probability
    \be
        \mathbb{P}(T_k\cap \Lambda_n^* \cap \Delta = \emptyset) = 1-\frac{|\Delta\cap T_k|}{|T_k|}.
    \ee
    Hence, we have
    \begin{equation}
        \mathbb{P}(\Delta \cap \Lambda_n^* = \emptyset) = \prod_{k\in \cS_m^d} \left(1-\frac{|\Delta\cap T_k|}{|T_k|}\right) \leq \exp\left(-\sum_{k\in \cS_m^d}\frac{|\Delta\cap T_k|}{|T_k|}\right),
    \end{equation}
  where we have used the inequality $1-x \leq e^{-x}$ for $x\geq 0$. Now we simply calculate
    \be
        \sum_{k\in \cS_m^d}\frac{|\Delta\cap T_k|}{|T_k|} = m^{d} \sum_{k\in \cS_m^d}|\Delta\cap T_k| = m^{d}|\Delta|>c' n |\Delta|,
    \ee
    where we have used \eref{pp}.
    Hence, for each fixed simplex $\Delta\in\tilde \Gamma$ we have
    \be
        \mathbb{P}(\Delta \cap \Lambda_n^* = \emptyset) < \exp(-c'n|\Delta|).
    \ee
    Finally, we apply a union bound over all simplices $\Delta\in \tilde\Gamma$.     The cardinality $\#(\tilde \Gamma)$ is bounded by $(4d^2n+1)^{d(d+1)}$, since this is the number of $(d+1)$ tuples of points in $\Lambda_r$.   This shows that
    \be
    \pi(C;n,d)\le \tilde\pi(C;n,d)<     [4d^2n + 1]^{d(d+1)}\exp\left(-c'c\frac{C}{2}\log(n)\right) < Ln^{A- \alpha C},
    \ee
    where $L,A$ and $\alpha$ depend only on $d$. This completes the proof.    

    \end{proof}

\begin{remark}
    We note that point sets satisfying Property 1, which are sometimes called Danzer sets in the literature, have also been constructed and analyzed in \cite{bambah1971problem,solomon2016dense,stefanescu2025maximal}. However, the constructions in these papers are much more complicated than our construction in this section. A precise estimate of the constant has also been obtained in \cite{arizmendi2016large}, however only for $d=2$. Let us also note that it appears to be an open problem whether the logarithmic factor in \eqref{Ssize} is required (see \cite{bambah1971problem,solomon2016dense}). If this factor could be removed, it would remove the logarithmic factors in Theorem \ref{T:further-improved}, removing the gap with the lower bounds.
\end{remark}
 
 \subsection{Proof of Theorem \ref{T:further-improved}}
 \label{SS:Prooftheorem}
 We fix $d$ and $\Omega=\Omega_d$,  fix $1\le p\le\infty$, and  fix $n\geq 2$. Let $\Lambda^*:=\Lambda_n^*$ be any set of $n$ points that satisfy Proposition \ref{property-1-proposition}.   We shall show that sampling at the data sites $\Lambda_n^*$ satisfies
 \be 
 \label{s*}
 \rho(\cL(\Omega),\Lambda_n^*)_{L_p} \le C\kappa(n,p,d),
 \ee
 for all $n$ and $1\le p\le \infty$.
 In order to prove \eref{s*}, 
 it suffices to consider any two functions 
$f,g\in\cL:=\cL(\Omega)$ that agree at the data sites in  $\Lambda^*$ 
and show that 
\be
\|f-g\|_{L_p}\leq C\kappa(n,p,d).
\label{goal}
\ee
Indeed, this shows that
any map that produces from the data vector $w=w(f)=(f(x^1), \ldots,f(x^n))$ a function
$g\in \cL$ that agrees with the data 
on $\Lambda^*$ will satisfy
\iref{goal}
and therefore prove the corresponding bound on $\rho_n(\cL)_{L_p}$. A constructive way to produce 
such a recovery map is proposed in  Section \S\ref{S:OR}.

Next, note that whenever $f,g\in\cL$, the function  $\max\{f,g\}$ is also in $\cL$, and we have
\be
\|f-g\|_{L_p}\leq \|f-\max\{f,g\}\|_{L_p}+\|g-\max\{f,g\}\|_{L_p}.
\ee
Therefore, we assume without loss of generality that 
\be 
\label{assume1}
g(x)\geq f(x), \quad x\in\Omega.
\ee

For the remainder of this section, we fix any $f,g\in \cL$ that satisfy \eref{assume1} and agree on the data sites $\Lambda^*$ and show that \eref{goal} is valid. For this purpose,  we introduce the distribution function
    \begin{equation}
    \label{mut}
        \mu(t) := |A_t|, \quad A_t:=\{x\in \Omega \;:\; g(x) - f(x) > t\}, \quad t\geq 0.
    \end{equation}
The following lemma gives the bound we shall prove for $\mu$.
\begin{lemma}  Let $f,g\in\cL$ satisfy \eref{assume1}.  Then, the following bounds hold for $\mu$:
\label{L:boundmu}
    \begin{equation}
        \mu(t) \lesssim \begin{cases}
            0, & t \geq Cn^{-1/d},\\
            t^{-d}\left(\frac{\log{n}}{n}\right)^2, & cn^{-2/d}(\log{n})^{2/d} < t < Cn^{-1/d},\\
            1, & t \leq c n^{-2/d}(\log{n})^{2/d},
        \end{cases}
        \label{3cases}
    \end{equation}
    where the constants $C$ and $c$  depend only on $d$.
   
\end{lemma}

Before proceeding to the proof of this lemma, let us show how this  lemma proves \eref{goal} and therefore proves Theorem \ref{T:further-improved}.   We have
\begin{equation}
    \label{Lp}
        \|g - f\|^p_{L_p(\Omega)} = p\int_0^\infty t^{p-1}\mu(t)dt \lesssim \left(\delta^{2p/d} + \delta^2\int_{c\delta^{2/d}}^{Cn^{-1/d}}t^{-d+p-1}dt\right), \quad \delta := \frac{\log{n}}{n},
    \end{equation}
    where the constant in $\lesssim$ depends on $d$ and $p$.
    Evaluating the latter integral gives
    \begin{equation}
    \label{Lp1}
        \int_{c\delta^{2/d}}^{Cn^{-1/d}}t^{-d+p-1}dt \lesssim \begin{cases}
           n^{-(\frac pd-1)}, & p > d,\\
            |\log \delta|, & p = d,\\
            \delta^{-2 + \frac{2p}d}, & p < d,
        \end{cases}
    \end{equation}
    where again, the constant in $\lesssim$ depends only on $d$ and $p$.
    Estimates  \eref{Lp} and \eref{Lp1}   give 
    \iref{goal}.

 \subsection{Proof of Lemma \ref{L:boundmu}}

The last statement in \iref{3cases} is trivial since $|\Omega|\le1$. To prove the first statement of the lemma, we  note that since the set $\Lambda_n^*$ satisfies Property 2,  and $f,g\in \cL$, we have 
 \be
 t \geq 2Cn^{-1/d} \implies \mu(t)=0,
 \label{boundt}
 \ee
where $C$ is the constant in \iref{cover}. Indeed, due to Property 2, 
for every $x\in \Omega$, we can find $x^i\in \Lambda^*$, such that 
$|x-x^i|\leq Cn^{-1/d}$. Since $f(x^i)=g(x^i)$ and $f,g\in \cL$, we have
$$
g(x)-f(x)=(g(x)-g(x^i))-(f(x)-f(x^i))\leq 2|x-x^i|\leq 2Cn^{-1/d},\quad x\in \Omega,
$$
and thus, if $x\in A_t\neq \emptyset$, then $t< 2Cn^{-1/d}$.
This observation also shows that 
\be
\|f-g\|_{L_\infty}\leq 2Cn^{-1/d},
\ee
which is the optimal decay rate for $\rho_n(\cL)_{L_\infty}$ that was already established in Theorem \ref{T:tensorL}.

In going further, we therefore focus on values of $t$ with   
\be 
\label{trange}
 cn^{-2/d}(\log{n})^{2/d} < t < Cn^{-1/d}.
 \ee
We fix such a value of $t$ and   bound  $\mu(t)=|A_t|$ by covering $A_t$ by a finite collection of sets 
 $S(x)$ defined as follows.  Given any $x\in \Omega$, we pick and fix a $u\in \partial g(x)$ and introduce the set
    \begin{equation}
    \label{Sx}
      S(x) := \{y\in \Omega: f(y) \leq g(x) + u\cdot (y - x)\}.
    \end{equation}
Here, we have suppressed the dependence of $S(x)$ on $u$ and all subsequent arguments do not depend on the particular choice of $u$. The mapping 
$y\mapsto g(x)+u\cdot (y-x)$ is any affine function lying below $g$ and interpolating $g$ at $x$.

   We begin by observing some important properties of the sets $S(x)$, $x\in\Omega$.  For any $x\in\Omega$, we have  $x\in S(x)$, $S(x)$ is closed,  and $S(x)$ is a convex set since the mapping $y\mapsto f(y)-u\cdot (y - x)$ is a convex function.
   Second, whenever  $g(x) > f(x)$, we have     
  \be
  f(y) < g(x) + u\cdot (y-x)\leq g(y),\quad  y\in \text{int}(S(x)),
  \ee 
  because $g$ is convex.
  Hence,  $\text{int}(S(x))$ doesn't contain any of the data sites $x^i$. Therefore, from Property 1, we have (see \eref{Ssize}),
    \begin{equation}\label{S-x-upper-bound-equation-1665}
        |S(x)| \leq C\frac{\log{n}}{n}, \quad x\in A_t.
    \end{equation}
    Our third observation is that if $x\in A_{t}$ and $y\in B(x,t/2)$, one has
 \be
 f(y) \leq f(x)+\frac t 2
 \leq g(x)-\frac t 2 \leq  g(x)+u\cdot(y-x),
 \ee
 where we have used the fact that $f$ and $g$ are both Lipschitz with constant $1$, and in particular that $|u|\leq 1$. Thus, we have
 \begin{equation}
B(x,t/2)\cap \Omega \subset S(x), \quad \hbox{when}\quad x\in A_t.
\label{ballinter}
\end{equation}
In particular, $x\in \text{int}(S(x))$.
 
Our strategy to estimate the size of $A_t$ is to cover it with sets $S(z^i)$ for properly chosen points $z^i\in A_t$. We proceed
in an inductive manner. Beginning with $A_t^0 := A_t$, we construct points $z^i$ as follows:
    \begin{enumerate}
        \item If $A_t^i = \emptyset$ the construction ends;
        \item Otherwise, choose any element $z^{i+1}\in A_t^i$;
        \item Set $A_t^{i+1} = A_t^i \setminus S(z^{i+1})
        =A_t\setminus(S(z^{1})\cup \cdots \cup S(z^{i+1}))$.
    \end{enumerate}  
    \vskip .1in
    \noindent
    {\bf Construction Claim:}  
{\it We claim that the above construction ends after $k$ steps with 
       \begin{equation}
        k \leq Ct^{-d}\frac{\log{n}}{n},
    \end{equation}
    where $C$  depends only on $d$.}
    
    \vskip .1in

    Assuming for the moment that the above claim is true, we have
\be
  A_t\subset \bigcup_{i = 1}^k S(z^i),
 \ee
 and thus by \eqref{S-x-upper-bound-equation-1665}, since $z^i\in A_t$,  it follows that 
 \be
 \label{mutA}
 \mu(t) = |A_t| \leq C k\frac {\log n}{n},
 \ee
 which finishes the proof of Lemma \ref{L:boundmu} and therefore, as we observed above it also finishes the proof of Theorem \ref{T:further-improved}.  Accordingly, the remainder of this section is devoted to proving the above construction claim.

 For this purpose, we will use the concept of the polar $S^\circ$ of a closed convex set
$S\subset \mathbb{R}^d$. When   $0\in \text{int}(S)$, $S^\circ$ is
defined as the convex set
\begin{equation}
    S^\circ := \{z\in \mathbb{R}^d : z\cdot y \leq 1~\text{for all $y\in S$}\}.
\end{equation}
A fundamental fact, first proved by Mahler \cite{mahler1939ubertragungsprinzip}, see also Remark~\ref{remarkmahler} below, is that 
\begin{equation}\label{mahlers-inequality}
    |S||S^\circ| \geq c,
\end{equation}
where $c$ depends only on $d$. For any $x\in A_t$, we consider the set
    \begin{equation}\label{definition-of-S-circ}
        S^*(x) := u + (t/3)(S(x) - x)^\circ = \{v\in \mathbb{R}^d:~(v-u)\cdot (y-x) \leq t/3~\text{for all $y\in S(x)$}\},
    \end{equation}
where $u\in \partial g(x)$ is the subgradient, used in the definition of $S(x)$, see \eref{Sx}. It follows from \eqref{mahlers-inequality} 
that 
    \begin{equation}
        |S^*(x)| = (t/3)^d |(S(x) - x)^\circ|
        \geq c(t/3)^d  |S(x)|^{-1}.
            \end{equation}
Combining this estimate with \eqref{S-x-upper-bound-equation-1665}, gives  that for $x\in A_t$, we have 
\be
\label{polar-lower-bound-1994}
|S^*(x)| \geq c t^{d}\frac n{\log n},
\ee
where $c$ depends only on $d$.
 The following technical lemma forms the crux of the proof of the construction claim. It will be used to bound $k$ in \eref{mutA} and explains the choice of $t/3$ in \eqref{definition-of-S-circ}.
    \begin{lemma}\label{empty-intersection-polar-lemma}
        If $x,x'\in A_{t}$, and $x'\notin S(x)$, then $S^* (x) \cap S^* (x') = \emptyset$. 
    \end{lemma}
    \begin{proof}
       Let $u\in \partial g(x)$ and $u'\in \partial g(x')$ denote the subgradients  which define $S(x)$ and $S(x')$, respectively. We will assume that there exists $v\in S^*(x)
\cap S^*(x')$ and reach a contradiction.

First, we observe that it suffices to work in the univariate case. Indeed, for the general case $d\geq 1$, let us  restrict our attention to the line $L$ determined by  $x$ and $x'$ (which are distinct since $x'\notin S(x)$). We pick a point $x_0\in L\cap\Omega$,
and set $e:=|x'-x|^{-1}(x'-x)$, so that  $L=\{y:\,y=x_0+se, \,s\in \R\}$.
Then, the convex univariate functions 
\be
\label{uni}
s\mapsto f(x_0+se)=:f_0
(s), \quad 
s\mapsto g(x_0+se)=:g_0(s),
\ee
which are defined on a bounded interval, satisfy the exact same assumptions
as $f$ and $g$, when $x$ and $x'$ are replaced by 
the points $s_0:=(x-x_0)\cdot e$ and $s_0':=(x'-x_0)\cdot e$, noting that  $u\cdot e$ and $u'\cdot e$ are subgradients of $f_0$ and $g_0$ at these points, respectively. Note  that $x=x_0+s_0e$, $x'=x_0+s'_0e$,  $g(x)=g_0(s_0)$, and  $g(x')=g_0(s'_0)$. 

If $v$ belongs to $S^*(x)\cap S^*(x')$, then it is easily seen that $v\cdot e$ belongs to its univariate counterpart $S^*(s_0)\cap S^*(s_0')$. Indeed, if $v\in S^*(x)$, then for all $y\in S(x)$,  we have
$(v-u)(y-x)\leq t/3$.
Note that 
\begin{eqnarray}
\nonumber
S(x)&=&\{y\in \Omega:\,f(y)\leq g(x)+u\cdot(y-x)\}
=
\{y\in \Omega:\,f(y)\leq g_0(s_0)+u\cdot(y-x_0-s_0e)\}
\end{eqnarray}
We can then    choose $y=x_0+se\in S(x)$, to obtain, 
$$
 \{x_0+se\in \Omega:\,f_0(s):=f(x_0+se)\leq g_0(s_0)+u\cdot e(s-s_0)\}\subset S(x).
$$
Moreover, since $s_0e=x-x_0$, we see that
$$
(v\cdot e-u\cdot e)(s-s_0)=
(v-u)(x_0+se-x)=(v-u)(y-x)\leq t/3,
$$
holds for all $s\in \{s:
x_0+se\in \Omega, \,\,f_0(s)\leq g_0(s_0)+u\cdot e(s-s_0)\}=S(s_0)$. Thus, $v\cdot e\in S^*(s_0)$. Similarly, one can show that if $v\in S^*(x')$, then $v\cdot e\in S^*(s')$.

In view of the above considerations,  we proceed by working in the case $d=1$, with
$\Omega$ being a  bounded closed interval, assuming that $x<x'$. The proof when $x'<x$ is similar. The sets $S(x),S(x'),S^*(x),S^*(x')$ are closed intervals. We also observe that, since $x'\notin S(x)$ and $x\in S(x)$,
the point $\t x:=\max S(x)$ satisfies $x<\t x<x'$. 
We also define $\o x:=\max\{x,\min S(x')\}<x'$, since $x'\in\text{int}(S(x'))$.
We then write
\be
\t x=\alpha x'+(1-\alpha)x \quad {\rm and} \quad 
\o x=\beta x+(1-\beta)x',
\ee
for some $0<\alpha<1$ and $0<\beta\leq 1$.
Note that $\o x=x$ and $\beta=1$ 
if and only if $x\in S(x')$.

Since $\t x$ is the right end point of $S(x)$, one has
$f(\t x)=g(x)+u(\t x-x)$, and therefore, by the convexity
of $f$,
\be
\label{eq1}
\alpha f(x')+(1-\alpha)f(x)\geq f(\t x)
=g(x)+u(\t x-x).
\ee
Next, using that $x,x'\in A_t$ and \eref{eq1}, we obtain
\be
\alpha g(x')+(1-\alpha)g(x)
\geq \alpha f(x')+(1-\alpha)f(x)+t
\geq g(x)+u(\t x-x)+t.
\ee
Using the convexity of $g$ and since $\t x-x=\alpha(x'-x)$,
we thus find that
\be
t\leq \alpha(g(x')-g(x))-
\alpha u(x'-x)\leq  \alpha(u'-u)(x'-x),
\ee
which reveals a lower bound
\be
u'-u\geq \frac t {\alpha(x'-x)}.
\label{lowerjump1}
\ee
When $\beta<1$, we  know that $\o x$ is the left endpoint  of $S(x')$, and we obtain by a similar 
reasoning that 
\be
t\le \beta (g(x)-g(x'))+ \beta u'(x'-x)\leq  \beta(u'-u)(x'-x).
\ee
We thus have the improved lower bound
\be
u'-u\geq \frac t{x'-x} \max\{\alpha^{-1},\beta^{-1}\}, 
\label{lowerjump2}
\ee
which also holds when $\beta=1$, since it then reduces
to \iref{lowerjump1}.

On the other hand, since $\t x\in S(x)$
and $\o x\in S(x')$, and in view of the definition
of $S^*(x)$ and $S^*(x')$, one has for any $v$ in their intersection the framing
\be
u'-\frac t {3\beta(x'-x)}=u'-\frac t {3(x'-\o x)}\leq v \leq u+\frac t {3(\t x-x)}=u+\frac t {3\alpha(x'-x)},
\ee
and therefore
\be
u'-u\leq \frac t{3(x'-x)}(\alpha^{-1}+\beta^{-1})
\leq \frac 2 3  \frac t{x'-x} \max\{\alpha^{-1},\beta^{-1}\},
\ee
which  contradicts \iref{lowerjump2}.
    \end{proof}

Since the construction process of the $z^i$'s ensures that 
$z^j\notin S(z^i)$ if $i < j$, and all $z^i\in A_t$, the above lemma implies that
\be
\label{emptyintersect}
S^*(z^i)\cap S^*(z^j) = \emptyset, \quad i\neq j.
\ee
On the other hand, we know from \iref{ballinter} that if $x\in A_t$ and 
if $B(x,t/2)\subset \Omega$, 
then $B(x,t/2)\subset S(x)$, or equivalently
\be
B(0,t/2)\subset S(x)-x.
\ee
It follows that 
\be
(S(x)-x)^\circ \subset B(0,t/2)^\circ =B(0,2/t).
\ee
By rescaling and translating by $u$ which has norm at most $1$ (since $g$ is $1$-Lipschitz),
we find that
\be
S^*(x)\subset u+B(0,2/3) \subset B(0,5/3).
\ee
This containment is no longer guaranteed for those $x\in A_t$ with distance less than $t/2$
from $\partial \Omega$. For such $x$, we will use the fact that there exists a larger
convex set $\t S(x)\supset S(x)$ such that 
\be
B(x,t/2)\subset \t S(x) \quad {\rm and} \quad |\t S(x)| = C |S(x)|,
\label{Stilde}
\ee 
where $C$  depends only on $d$. We postpone the construction of $\t S(x)$ to the end of the proof, and set 
$\t S^*(x):=u+(t/3)(\t S(x)-x)^\circ$,
so that we have the containment
\be
\tilde S^*(x) \subset B(0,5/3).
\ee
By convention,  we set $\t S(x):=S(x)$ for those $x\in A_t$ such that $B(x,t/2)\subset \Omega$.
Note that for all $x\in A_t$ we have the upper bound
\be
|\t S(x)|\leq C \frac{\log n}{n},
\ee
and therefore the lower bound
\be
|\tilde S^*(x)| \geq ct^{-d}\frac {n}{\log n},
\ee
where $C$ and $c$  depend only on $d$.
Since we have $\tilde S^*(z^i) \subset S^*(z^i)$ for all $i=1,\dots,n$, we find 
by \iref{emptyintersect} that we also have
\be
\label{emptyintersecttilde}
\t S^*(z^i)\cap \t S^*(z^j) = \emptyset, \quad i\neq j,
\ee
and therefore 
\be
ck t^{d}\frac {n}{\log n}\leq \sum_{i=1}^k |\t S^*(z^i)| \leq C:=|B(0,5/3)|.
\ee
This yields an upper bound for $k$, namely,
       \begin{equation}
        k \leq Ct^{-d}\frac{\log{n}}{n},
    \end{equation}
    where $C$  depends only on $d$.

    We conclude the proof by explaining the construction of $\t S(x)$ from $S(x)$ in the case
    where $x\in A_t$ is at distance less than $t/2$ from $\partial \Omega$. We first observe that,
    trivially $\|f-g\|_{L_\infty}\le 2$, so that     
    we may always assume that $t\leq 2$.
    Since $x\in \Omega$, 
    we can always find a $x'\in \Omega$ such that 
    \be
   |x-x'|=\min\{1,t\}/4 \quad{\rm and}\quad {\rm dist}(x',\partial\Omega) \geq \frac{\min\{1,t\}}{4\sqrt d}\ge \frac{t}{8\sqrt{d{}}}:=ct.
   \label{xprime}
   \ee
Observe that 
        \begin{equation}
            B(x',ct)\subset B(x,t/2)\cap \Omega\subset S(x).
            \label{contxprime}
        \end{equation}
 We then define the desired set by
         \begin{equation}
            \t S(x) := x' + \frac{3}{4c}(S(x) - x')=x' + 6\sqrt d(S(x) - x'),
        \end{equation}
        which obviously contains $S(x)$ since $6\sqrt d\geq 1$ and whose 
         size is
        \be
        |\t S(x)|=C|S(x)|, \quad C:=(6\sqrt d)^d.
        \ee
 Finally, we find  from \iref{contxprime} that $B(0,ct)\subset S(x) - x'$, so that
 $B(x',3t/4)\subset \t S(x)$, and therefore
 \be
 B(x,t/2)\subset B(x',3t/4) \subset   \t S(x),
 \ee
 which confirms \iref{Stilde} and ends the proof of the construction claim.

 \begin{remark}
 \label{remarkmahler}
 In \cite{mahler1939ubertragungsprinzip}, Mahler apparently only proved \eqref{mahlers-inequality} for convex and centrally symmetric bodies. However, it follows easily from this that the same inequality holds for closed convex bodies containing $0$ in their interior. Indeed, suppose that $S$ is a closed convex body with $0\in \text{int}(S)$. Let $z^1$ be the  point in $S$ with a largest norm $|z^1|$, and choose sequentially points $z^i\in S$ whose distance from $\text{span}(z^1,...,z^{i-1})$ is maximized. Then $S$ contains the simplex $\Delta$ with vertices $z^0 = 0,z^1,...,z^d$, whose volume is
\begin{equation}
    |\Delta|=\frac{1}{d!}\prod_{i=1}^d\dist(z^i,\text{\normalfont span}(z^1,...,z^{i-1})),
\end{equation}
and therefore $|S|\geq |\Delta|$.
On the other hand, the symmetric convex hull 
$$
S_{\rm sym}:=\text{conv}(S\cup -S) \subset \cR
$$
is contained in a rectangle $\cR$ with side lengths $l_i = 2\dist(z^i,\text{span}(z^1,...,z^{i-1}))$, and thus
$$
2^dd!|\Delta|=2^d\prod_{i=1}^d\dist(z^i,\text{\normalfont span}(z^1,...,z^{i-1}))=|\cR|\geq |S_{\rm sym}|
$$
Hence, it follows that $|S| \geq \t c|S_{\rm sym}|$ 
with $\t c:=(2^d d!)^{-1}$. On the other hand, since $S\subset S_{\rm sym}$,  it follows from the definition that $S^\circ \supset S_{\rm sym}^\circ$. Hence
\begin{equation}
    |S||S^\circ| \geq \t c|S_{\rm sym}||S_{\rm sym}^\circ| \geq  \t cc,
\end{equation}
where $c$ is the constant from \cite{mahler1939ubertragungsprinzip} for convex and centrally symmetric bodies. 
We note that determining the optimal constant $c$, both in the case of general or centrally symmetric convex sets, is called Mahler's conjecture, and is a well-known and exceptionally difficult open problem, which was solved by Mahler only in the case $d = 2$.
\end{remark}

\section{Sharp upper bounds for the sampling numbers of $\cB$}
\label{S:optimalB}

In this section, for each fixed $d$, we show that the estimates in Theorem \ref{T:further-improved} for the class $\cL:=\cL(\Omega)$,  $\Omega:=\Omega_d$, are also valid for the
larger class $\cB:=\cB(\Omega)$. Namely, the following theorem holds.

\begin{theorem}\label{T:further-improvedB}
    Let $d \geq 2$ and $1 < p < \infty$. Then,  for any $n \geq 2$,  we have
    \begin{equation}\label{improved-upper-bound-equationB}
    \rho_n(\cB)_{L_p} \leq C \kappa(n,p,d),
\end{equation}
where $\kappa(n,p,d)$ is defined in  \eref{improved-upper-bound-equation} and $C$ depends only on $d$ and $p$.
\end{theorem}

The main difference between Theorem \ref{T:further-improved} and Theorem \ref{T:further-improvedB} lies in the sampling strategy, which is more delicate for the class $\cB$. We need to refine the grid near the boundary $\partial\Omega$ in a  similar way to how it was done in the case of a tensor product grid studied in \S \ref{S:TensorB}.
Let us also recall that for the $L_\infty$ norm one has
$\rho_n(\cB)_{L_\infty} \asymp 1$,  while 
$\rho_n(\cL)_{L_\infty} \asymp n^{-1/d}$, and the case $p=1$ has already been discussed in Theorem 
\ref{T:tensorB}, see \eref{orBgen}, where there are no logarithms present.

\begin{proof}
Our strategy  is to 
decompose $\Omega$ into
subdomains and apply Theorem \ref{T:further-improved} to each of them 
after a proper rescaling  that will place us  in the class $\cL$  for each subdomain.  We fix the values of $n\ge 1$, and $d\ge 1$ and $1<p<\infty$.

We first define the partition in the univariate case.
We set 
$$
J_0:=[1/3,2/3],
$$
and with $a_l:=\frac 1 3 2^{-l}$, $l=0,1,\dots$, we define
\be
J_{-l}:=\frac 1 3 \cdot [2^{-l},2^{-(l-1)}]=:[a_l,a_{l-1}] \quad 
{\rm and}\quad J_{l}:=[1-a_{l-1},1-a_{l}],
\quad l=1,2,\dots .
\ee

The  collection $\{J_l\}_{l\in\Z}$  forms a   
partition of $(0,1)$ whose interiors are pairwise disjoint.  
By tensorization we obtain a partition $\{R_k\}_{k\in \Z^d}$ of  the interior of  $\Omega$ where
\be
R_k:=J_{k_1}\times\cdots \times J_{k_{d}},
\quad k=(k_1,\dots,k_d)\in\Z^d.
\ee
\vskip .1in
\noindent
\vskip .1in
\noindent
Since $J_l$ and $J_{-l}$ have length $\frac 13 2^{-l}$, the rectangle $R_k$ has measure
\be
\label{Omkvol}
|R_k |=3^{-d}\,2^{-(|k_1|+\cdots+|k_d|)}.
\ee
 We next define $\phi_l$ as the univariate affine function that maps $[0,1]$ onto $J_l$, 
Therefore, for all $k\in \Z^d$,   the transformation
$\Phi_k:\Omega\to R_k$, defined as
\be
x\mapsto  \Phi_k(x):=(\phi_{k_1}(x_1),\dots,\phi_{k_d}(x_d)), \quad x=(x_1,\dots,x_d),
\ee
is a bijective map between $\Omega$ and $R_k$.

For any given $n$, we  build a sampling set $\Lambda_n^*$
of cardinality at most $n$ as follows. For each $k\in\Z^d$, 
we  define
\be
n_k:=\lfloor cn2^{-\alpha (|k_1|+\cdots+|k_d|)}\rfloor, \quad \hbox{where}\quad c:=\(\frac {2}{1-2^{-\alpha}} -1\)^{-d}, 
\ee
and  where $\alpha>0$   is fixed and   small enough to ensure that
\be
\label{condalpha}
1-\alpha p\t r/d>0.
\ee
We pick sets $\Lambda^*_{n,k}$ comprised of $n_k$ points in $\Omega$ according to Proposition \ref{property-1-proposition},
and therefore satisfying Property 1 and 2 with $n_k$ in place of $n$. Note that, in view of the definition of $n_k$, 
the set $\Lambda^*_{n,k}$ is
non-empty only for finitely many values of $k$.
Using $\Phi_k$, we then map these points 
into $R_k$ and define
\be
\label{lambdaB}
\Lambda_n^*:=\bigcup_{k\in\Z^d}%
\Phi_k(\Lambda^*_{n,k}).
\ee
The set $\Lambda_n^*$ has cardinality at most $n$ since
 \be
\#(\Lambda_n^*) =\sum_{k\in\Z^d} n_k \leq cn \sum_{k\in\Z^d}2^{-\alpha (|k_1|+\cdots+|k_d|)}
=cn\(\sum_{l\in\Z} 2^{-\alpha|l|}\)^d
=cn\(\frac {2}{1-2^{-\alpha}} -1\)^d=n.
\ee

We now consider any two functions $f,g\in \cB$ that agree on the data sites $\Lambda^*_n$.
Similarly to the proof of Theorem \ref{T:further-improved}, we need to show that for some $C$ large enough
\be
\label{toprove22}
\|f-g\|_{L_p(\Omega_d)}\leq C\kappa(n,p,d).
\ee

In order to prove \eref{toprove22}, we consider any $d$ dimensional rectangle $R_k$, $k\in\Z^d$ and derive a bound for $\|f-g\|_{L_p(R_k)}$.  We fix $R_k$ and first   observe that for any $f\in \cB$, upon applying   \iref{uplow1}
to the univariate convex functions 
\be
t\mapsto f(x_1,\dots,x_{j-1},t,x_{j+1},\dots,x_d),
\ee
the coordinates $u_j$ of 
any $u=(u_1,\dots,u_d)\in \partial f(x)$
satisfy the bounds
\be
x\in R_k \implies  |u_j| \leq 6\cdot 2^{|k_j|}, \quad j=1, \ldots,d,\quad \hbox{with}\quad u=(u_1,\dots,u_d)\in \partial f(x).
\label{boundj}
\ee
For a smooth function, this simply means that $\Big |\frac{\partial f}{\partial x_j}(x)\Big | \leq 6\cdot 2^{|k_j|}$, $j=1, \ldots,d$.

These bounds  hold for the subgradients of $f$ and $g$. Given $f, g\in \cB$,
 for each $k\in \Z^d$, we denote by $f_k$ and $g_k$ 
the pullbacks of $f|_{R_k}, g|_{R_k}$,
\be
f_k(x):=(2\sqrt d)^{-1}f(\Phi_k(x)) \quad {\rm and} \quad g_k(x):=(2\sqrt d)^{-1}g(\Phi_k(x)), \quad x\in \Omega, \quad k\in \Z,
\ee
which are now functions defined on $\Omega$. Since $|\phi_l'(t)|= 2^{-l}/3$, 
it follows from \eref{boundj} that each component of their subgradient is  bounded 
by $(\sqrt d)^{-1}$. Therefore,
 $f_k,g_k\in \cL$, and agree on the data sites $\Lambda^*_{n,k}$ if $f$ and $g$ agree on $\Phi_k(\Lambda^*_{n,k})$. From Theorem \ref{T:further-improved}, we find that
\be
\|f_k-g_k\|_{L_p(\Omega)}\leq C \kappa(n_k,p,d).
\label{goalk}
\ee
This estimate is only for the value $n_k\geq 2$, but in the case $n_k=0,1$, we simply have $\|f_k-g_k\|_{L_p(\Omega)}\leq C$ since $f_k$ and $g_k$ are bounded by $1$. 
This allows us to rewrite,
regardless of the value of $n_k$,
\be
\|f_k-g_k\|_{L_p(\Omega)}\leq
C(\max\{1,n_k\})^{-\t r/d} \left [\log(\max\{2,n_k\})\right ]^{\gamma},\quad \t r:=\min\{2,1+d/p\},
\ee
where $\gamma$ is the power of the logarithmic factor showing on the right side of \eref{improved-upper-bound-equationB}, and $C$ depends only on $d$ and $p$. We note that
\be
\max\{1,n_k\} \geq \frac 1 2cn2^{-\alpha (|k_1|+\cdots+|k_d|)}
\quad {\rm and}
\quad \log(\max\{2,n_k\})\leq \log n,
\ee
where the first inequality follows from the definition of $n_k$
and the second uses $n_k\leq n$ and $n\geq 2$, and therefore
\be\label{final}
\|f_k-g_k\|_{L_p(\Omega)}\leq
C2^{\alpha \t r/d(|k_1|+\cdots+|k_d|)}n^{-\t r/d}(\log n)^\gamma, \quad n\geq 2.
\ee.

The $L_p$ norm of $(f-g)$ can then be decomposed according to
\be
\|f-g\|_{L_p(\Omega)}^p=\sum_{k\in\Z^d} \|f-g\|_{L_p(R_k)}^p=(2\sqrt d)^p\sum_{k\in\Z^d} |R_k|\|f_k-g_k\|_{L_p(\Omega)}^p.
\ee
We then obtain from \eref{final},  \eqref{Omkvol} and the fact that $\beta :=1-\alpha p\t r/d>0$, the estimate
\be
\|f-g\|_{L_p}^p \leq Cn^{-p\t r/d}(\log n)^{p\gamma} \sum_{k\in\Z^d}2^{-\beta(|k_1|+\cdots+|k_d|)} 
= C \(\frac 2{1-2^{-\beta}}-1\)^d n^{-p\t r/d}(\log n)^{p\gamma},
\ee
where $C$ depends only on $d$ and $p$. This concludes the proof.
\end{proof}

\section{Recovery algorithms for the model classes $\cL$ and $\cB$}
\label{S:OR}
 In the previous sections of this paper,  we have determined the asymptotic behavior (up to logarithm factors) of the sampling rates $\rho_n(\cK)_{L_p}$ for the model classes $\cK=\cL, \cB$ when the error is measured in $L_p$ for $1\le p\le \infty$.   For those results, the key issue was where to sample the functions.   Given  $n\ge 1$, we have described a set $\Lambda_n^*\subset\Omega=\Omega_d$ of data sites so that sampling an $f\in \cK$ on $\Lambda_n^*$  gives a data vector $w(f)$,    which determines  $f$  to the prescribed accuracy in $L_p$. 
 
 We then  noted that any mapping $A: \R^n \to L_p$ that maps each data vector $w=w(f)$ with $f\in \cK$, to an element in the set 
 \be 
 \label{Kw1}
 \cK_w:=\{g\in \cK: w(g)=w\},
 \ee
  gives a recovery algorithm with the near optimal  performance of Theorem 4.1 and Theorem 5.1.  In this section, we discuss how to numerically  construct such a mapping $A$.

  We shall work in the following setting.   The model class $\cK$ will always be either $\cL=\cL(\Omega)$ or $\cB=\cB(\Omega)$ and
the set  $\Lambda=\{x^i,\ i=1,\dots,n\}\subset \Omega$ of data sites can be arbitrary. 
 We  shall construct a practical recovery algorithm  $A_\Lambda^\cK:\R^n\to L_\infty(\Omega)$   which, when  
applied to any data vector $w=w(f)$
with $f\in \cK$,  
constructs a function 
\be
\label{defAK}
A_\Lambda^\cK(w) =g_w \in \cK_w.
\ee
This guarantees that $A$ is near optimal in the sense that for each $1\le p\le\infty$, we have
\be
\label{no1}
\|f-A_\Lambda^\cK(w(f))\|_{L_p}\le 2\rho(\cK,\Lambda)_{L_p}, \quad f\in \cK.
\ee

 Let us remark on two noteworthy properties of the  $A_\Lambda^\cK$ we construct.  The first is that it will not depend on the value $p$ of the $L_p$ space where we measure error.  In addition,  $A_\Lambda^\cK(w)$ will be defined for any $w\in\R^n$, not just data vectors.  This will be useful when treating noisy data. 
 When  $\Lambda=\Lambda_n^*$, the algorithm $A^\cK_{\Lambda_n^*}$  provides, via \eref{no1}, the rates obtained in Theorem~\ref{T:further-improved} (for $\cK=\cL$) and Theorem~\ref{T:further-improvedB} (for $\cK=\cB$) when the data is not noisy.  

Let us note that in  optimal recovery in Banach spaces, constructing a recovery operator is well studied, but there is no  numerical procedure that is known to work in a general setting.
In our case, we are successful because each of the model classes $\cL$ and $\cB$ has the following property:  
\vskip .1in
\noindent
{\bf Supremum Property}: {\it Each of the classes $\cK=\cL,\cB$ is closed under suprema. More precisely, if $\cM$ is a subset of $\cK$,  
then the function $g_\cM$ defined by
\begin{equation}
        g_\cM(x) := \sup_{g\in \cM} g(x),  \quad x\in\Omega,
    \end{equation}
also belongs to $\cK$. }  
\vskip .1in

As a particular consequence of this property, let us note that if $w=w(f)$ for some $f\in \cK$, then the function 
\begin{equation}
\label{recover-functions-definition}
   g_w^\cK(x) := \sup_{g\in \cK_w} g(x)\in \cK,
\end{equation}
 and even belongs to $\cK_w$ because $g_w^\cK(x^i)=w_i$ for all $i=1,...,n$.   In fact,  $g^\cK_w$ is the largest function in $\cK_w$. We  show below that this function can be constructively  found by using classical optimization  techniques.  

We now proceed  to describing our proposed constructive recovery mapping.   Our definition of $A_\Lambda^\cK$ is based, in part, on the fact that any function $g\in \cK$ only takes values in the interval $[-1,1]$.    
In particular, if  $V$ denotes the set of the $2^d$ vertices of $\Omega$, then for any $g\in \cK$, we have
\be 
\label{O1}
g(v)\le 1,\quad v\in V.
\ee

 Recall that we want the recovery map $A_\Lambda^\cK$ to be defined for any vector in $\R^n$.  Given any $w\in\R^n$, we define
\be
\label{defwtilde}
\tilde w:= (\tilde w_1,\dots,\tilde w_n),\ {\rm where}\ \tilde w_i:=\max(-1,w_i),\ i=1,\dots,n,
\ee
and introduce the sets
 \begin{equation}
 \label{deftildeKw}
        \widetilde \cK_w := \{g\in \cK:~g(x^i) \leq \tilde w_i, \ x_i\in\Lambda,   \ {\rm and}\ g(v)\le 1, \ v\in V\},\quad w\in\R^n.
    \end{equation}
    Note that $\tilde \cK_w$ is non-empty because it contains the function $g\equiv -1$.
    We define our reconstruction mapping for any $w\in\R^n$ by
 \begin{equation}
\label{defAopt}
        A_\Lambda^\cK(w):= \tilde g_w^\cK, \qquad \textrm{where for }x\in \Omega \qquad \tilde g_w^\cK(x) := \sup_{g\in \widetilde \cK_w} g(x).
    \end{equation}
    This mapping is well defined for all $w\in\R^n$  and due to the supremum property, the function $\tilde g_w^\cK$ is always in $\cK$.

The following lemma shows that  $\tilde g_w^\cK$  coincides with the previously defined $g_w^\cK$ when $w=w(f)$ for a function in $f\in \cK$, and therefore $A_\Lambda^\cK$ is near optimal in the sense of \eref{no1}.

\begin{lemma}
\label{L:Aopt}
    Let $\cK$ be either of the model classes $\cL$ or $\cB.$ Suppose that $f\in \cK$   and $w = w(f)$. Then,  we have
    \begin{equation}\label{modified-opt-problem}
        \tilde g_w^\cK= g_w^\cK,
    \end{equation}
 and therefore  $g_w^\cK\in \cK_w$ is a near optimal recovery  of $f$ from the data $w$.
\end{lemma}

\begin{proof}
Let us first observe that when $w=w(f)$, for some $f\in \cK$,  then $w_i=\tilde w_i$. 
    Since we clearly have $\cK_w\subset \widetilde \cK_w$,
    it follows that
   $g_w^\cK\le \tilde g_w^\cK$.
    On the other hand, each $g\in \widetilde \cK_w$ is in $\cK$ and satisfies $g(x^i)\le \tilde w_i=w_i$, $i=1,\dots,n$.   It follows that for any $g\in \widetilde \cK_w$, the function  $\bar g:=\max(f,g) \in \cK_w$ and $\bar g\ge g$.   Therefore, $\tilde g_w^\cK\le g_w^\cK$.
\end{proof}

 \subsection{Evaluating  $A_\Lambda^\cK(w)$, $w\in\R^n$}
 \label{SS:constructing}
We have proven in Lemma \ref{L:Aopt} that the mapping $A_\Lambda^\cK$ defined by \eref{defAopt} is a near optimal recovery and therefore achieves the  accuracy
$\rho(\cK,\Lambda)_{L_p}$, $1\le p\le\infty$, for the model classes  $\cK=\cL,\cB$, regardless of the choice $\Lambda$  of data sites. At this stage, we cannot claim that $A_\Lambda^\cK$ is practical since we have not described how one can numerically
    execute the optimization in \eref{defAopt}.  In this subsection, we simplify the optimization problem \eqref{defAopt} that needs to be solved by reducing the set $\widetilde \cK_w$ over which maximization takes place.  Namely, 
we next show that   the function $\tilde g_w^\cK=A_\Lambda^\cK(w)$ can be found by solving 
an optimization problem over the finite dimensional space of affine functions on $\Omega$.

 Let $\mathbb {L}$ denote the linear space of affine functions $\ell$ on $\R^d$, i.e., 
\be
\label{affine}
\mathbb {L}:=\{\ell:\,\,\ell(\circ)= \zeta_\ell\cdot \circ +b_\ell, \quad \zeta_\ell\in\R^d, \,b_\ell\in \R\}.
\ee
In the next lemma, we show that we  can replace the sets $\widetilde \cK_w$  in \eref{defAopt} by 
the following set of affine functions
\begin{equation}
    {\mathbb{L}}_w^\cL := \{\ell\in \mathbb{L}: |\zeta_\ell|\le 1 \ {\rm and} \  \ \ell(x^i)\le \tilde w_i,\ i=1,\dots,n ,\    {\rm and} \  \ell(v)\le 1,\ v\in V\},
    \label{LBw}
\end{equation}
and 
\begin{equation}
    {\mathbb{L}}_w^\cB :=  \{\ell\in \mathbb{L}: \  \ \ell(x^i)\le \tilde w_i,\ i=1,\dots,n, \    {\rm and} \  \ell(v)\le 1,\ v\in V\},
    \label{LLw}
\end{equation}
when $\cK=\cL$ and $\cK=\cB$, respectively.
Both of these sets are  defined for all $w\in\R^n$ and are non-empty since they contain the affine function $\ell\equiv -1$.
 For each $w\in \R^n$ and $\cK=\cL,\cB$, we define
 \be
 \label{defellK}
  \ell_w^\cK(x):=\sup_{\ell\in \mathbb{L}_w^\cK} \ell(x),\quad x\in\Omega.
  \ee

\begin{lemma}
   Let $\cK$ be either of the model classes $\cL$ or $\cB$. For any $w\in \R^n$ and any $x\in\Omega_d$, we have
    \be
    \label{L:newaffine}
       \tilde g^\cK_w (x) =  \ell_w^\cK(x).
       \ee
\end{lemma}
\begin{proof} Fix  any $w\in\R^n$ and fix any $x\in \Omega$.   Let $g\in \tilde \cK_w$ and let $\zeta$ be any element in $\partial g(x)$. Then the affine function $\ell(y):=g(x)+\zeta\cdot(y-x)$ satisfies
$\ell(y)\le  g(y)$, $y\in\Omega$, and so $\ell\in {\mathbb {L}}_w^\cK$ because $|\zeta|\le 1$   when $\cK=\cL$.   Since $g(x)=\ell(x)$, it follows that
\be 
\label{in1}
\tilde g_w^\cK(x):=\sup _{g\in \tilde \cK_w} g(x)\le \sup_{\ell\in {\mathbb{L}}_w^\cK}\ell(x)=: \ell_w^\cK(x).
\ee

    To prove the reverse inequality, let $w\in\R^n$ and $\ell\in \mathbb{L}_w^\cK$, and consider the function $g_\ell:=\max\{\ell, -1\}$.  Then, $g_\ell$ is a convex function which is in
    $\widetilde \cK_w$.  Since $\ell \le g_\ell$ on $\Omega$, we have
\be 
\label{in12}
\ell_w^\cK(x):=  \sup_{\ell\in {\mathbb{L}}_w^\cK}\ell(x)\le  \sup _{g\in \tilde \cK_w} g(x)=:\tilde g_w^\cK(x),\quad x\in\Omega .
\ee
The last two inequalities prove the lemma.
\end{proof}

The above lemma shows that  for each $w\in\R^n$, we have
\be 
\label{Copt}
A_\Lambda^\cK(w)=\ell_w^\cK, \quad \hbox{where}\quad \ell_w^\cK(x):=
\sup_{\ell\in 
{\mathbb L}^\cK_w} \ell(x),\quad x\in \Omega, \quad \cK=\cL,\cB.
\ee
Moreover, for each $x\in\Omega$, $\ell^\cK_w(x)$ can be found by classical optimization. For the model class $\cK=\cB$, $\ell_w^{\cB}(x)$ is found by linear programming, and when 
 $\cK=\cL$, $\ell_w^{\cL}(x)$ is found by optimization algorithms for conic programming due to the quadratic constraint on $\zeta$,
see, for example \cite{Boyd-Vandenberghe-convex-optimization-2004}.

 The recovery algorithm $A_\Lambda^\cK$ can be applied to any $w\in\R^n$ and gives a near optimal recovery when $w=w(f)$ for $f\in \cK$.  
 The following theorem shows that this mapping is stable.

\begin{theorem}
\label{T:stableA}
Let $\cK$ be either of the model classes $\cL$ or $\cB$.  For any $w,\eta\in \R^n$, one has the stability estimates
\be
\label{stableA}
\|A_\Lambda^\cK(w)-A_\Lambda^\cK(\eta)\|_{L_\infty}
\leq \|w-\eta\|_{\ell_\infty}.
\ee
\end{theorem}

\begin{proof}
We define $\delta:=\|w-\eta\|_{\ell_\infty}$.  For each  $i=1, \ldots,n$, we have
$$
\t w_i:=\max\{-1,w_i\}\geq \max\{-1,\eta_i-\delta\}= \max\{-1,\t \eta_i -\delta\}. 
$$
  We claim that for each $g\in \widetilde \cK_\eta$, the function
\be
\label{defbarg}
\bar g:= \max\{-1,g-\delta\}\in \widetilde \cK_w.
\ee
 Indeed, this function is convex,  takes values in $[-1,1]$, and in the case $\cK=\cL$ has its subgradients $u\in\partial \bar g(x)$, $x\in\Omega_d$, satisfying $|u|\le 1$.  Moreover, 
\be 
\bar g(x_i)=\max\{-1,g-\delta\}(x^i)\leq \max\{ -1,\t \eta_i-\delta\}\leq \t w_i, 
\ee
thereby verifying our claim.
Therefore, for every $g\in \widetilde \cK_\eta$, we have
$$
g-\delta\leq \bar g\leq \tilde g_w^\cK,
$$
where we have used the definition of $\tilde g_w^\cK$. The latter inequality 
implies that $\tilde g_w^\cK\geq \tilde g_\eta^\cK-\delta$. Similarly, we find 
$\tilde g_\eta^\cK\geq \tilde g_w^\cK-\delta$, which proves the result.
\end{proof}

 The following corollary, bounds the performance of $A$ when the  measurements of $f$ are noisy.
\begin{cor}
\label{C:noist}
Let $f\in \cK$ where $\cK=\cL$ or $\cB$.  Suppose that in place of the true measurement vector $w=w(f)$ we have the noisy measurement vector $\eta$ with
$\|w-\eta\|_{\ell_\infty}\le \delta$. Then, we have
\be
\label{noisyerror}
\|f-A_\Lambda^\cK(\eta)\|_{L_p}\le 2\rho(\cK,\Lambda)_{L_p}+\delta.
\ee
\end{cor}
\begin{proof}
This follows directly from Theorem \ref{T:stableA}.
\end{proof}

\section{Concluding remarks}
\label{S:widths}

In this paper, we have determined the rate of decay of the sampling numbers 
$\rho_n(\cK)_{L_p}$ for the model classes $\cK=\cL$ and $\cK=\cB$, through
upper and lower bounds that match up to logarithmic factors. 
We have also established
upper bounds on the linear sampling numbers $\rho_n^{\ell}(\cK)_{L_p}$ 
with matching lower bounds only in 
the case $p=1,2,\infty$. 
In the case $p=2$,  the rate of decay of 
$\rho_n^{\ell}(\cK)_{L_2}$ is inferior to that of $\rho_n(\cK)_{L_2}$. In contrast to classical model classes, this reveals an instance  where non-linear recovery methods from point samples significantly outperform their linear counterparts.
One remaining open question is the complete determination of the rate of decay of $\rho_n^{\ell}(\cK)_{L_p}$ for values of 
$p\neq 1,2,\infty$.

We have considered classes of functions defined on the unit cube
$\Omega=[0,1]^d$. A legitimate question is the validity of our results
for classes of convex functions,  defined on more general domains $\Omega\subset \R^d$. We leave the extension of our results to this more general setting for future work.

As observed in the introduction of this paper, sampling 
numbers have natural connections with Kolmogorov widths 
and entropy numbers. These quantities have been studied for the two model classes $\cL$ and $\cB$.

Unfortunately, the Kolmogorov widths of $\cL$ and $\cB$ are not completely known.
For $\cL$, it was proven  in  \cite{KM} that for  any dimension $d$, 
\be
\label{d1}
d_n(\cL)_{L_1}\asymp  n^{-2/d}.
\ee
We have also seen in the proof of Theorem \ref{T:tensorL} that
\be
\label{d2}
d_n(\cL)_{L_2}\asymp  n^{-\frac 3 {2d}}.
\ee
In view of the comparison \eref{widthsamp}, 
our results on linear sampling numbers
in Theorem \ref{T:tensorL} give the
general upper estimate
\be
\label{KolLup}
d_n(\cL)_{L_p}\leq C  n^{-r/d}, \quad 1\leq p\leq \infty,\quad\hbox{where}\quad r:=1+\frac 1 p, 
\ee
which is sharp when $p=1$ and $p=2$.

For the class $\cB$, in the univariate case $d=1$, the rate of decay of the Kolmogorov widths has been determined for all  $1\leq p<\infty$ in Theorem 1.2 from \cite{Konovalov2005}, 
where it was shown that (again with $r=1+1/p$), we have
\be\label{gen}
d_n(\cB)_{L_p}\asymp \begin{cases}
n^{-r}, \quad 1\leq p\leq 2,\\
n^{-\frac{3}{2}}, \quad \quad 2\leq p<\infty.
\end{cases}
\ee
Our result in Theorem \ref{T:tensorB} gives the
general upper estimate in any dimension $d\geq 1$
\be
\label{KolBup}
d_n(\cB)_{L_p}\leq C  n^{-r/d}, \quad 1\leq p<\infty,
\ee
which is again sharp when $p=1$ and $p=2$ since $\cL\subset \cB$, see \eref{d1} and \eref{d2}.

The rates of decay of the entropy numbers of the two classes $\cL$ and $\cB$ in $L_p$ are known. The first result on entropy of the class $\cL$ and all $d\geq 1$  is \be
\label{Bronstein}
\varepsilon_n(\cL)_{L_\infty}\asymp n^{-2/d},
\ee
which was shown in \cite{B}, Theorem 6. Later,  in \cite{KM}, it was proven that the entropy numbers satisfy
\be\label{KM}
\varepsilon_n( \cL)_{L_1}\asymp n^{-2/d}.
\ee
From this, one easily obtains that for all $1\leq p\leq \infty$,
\be\label{entropyall}
\varepsilon_n( \cL)_{L_p}\asymp n^{-2/d}.
\ee
Finally, it was proven in \cite{GS}
that the same result holds for the larger class $\cB$
\be\label{entr}
\varepsilon_n( \cB)_{L_p}\asymp n^{-2/d}, \quad 1\leq p<\infty.
\ee
Here $p=\infty$ is excluded because, as noted in \cite{GS}, the class $\cB$ is not compact in $L_\infty$.

We have mentioned that sampling numbers are not expected to decay faster than the entropy numbers. It is interesting to note that our results from \S \ref{S:optimal} and \S \ref{S:optimalB} 
show that both $\rho_n(\cL)_{L_p}$ and
$\rho_n(\cB)_{L_p}$ decay like 
$n^{-2/d}$ (up to logaritmic factors) when $1\leq p\leq d$, and therefore, up to logarithmic factors, provide the rate of decay of  the entropy numbers for $\cL$ and $\cB$.

\bibliographystyle{plain}

\bibliography{refs}

\Addresses

\end{document}